\documentclass[reqno]{amsart}
\usepackage[a4paper]{geometry}
\usepackage[T1]{fontenc} 
\usepackage[utf8]{inputenc}
\usepackage{amssymb,bm}
\usepackage[mathscr]{eucal}
\usepackage{mathrsfs}
\usepackage[english]{babel}
\usepackage{mathtools}
\usepackage{microtype}
\usepackage{enumerate,paralist}
\usepackage{xcolor}
\usepackage{graphicx}
\usepackage[colorlinks=true]{hyperref}
\hypersetup{urlcolor=blue,citecolor=red,linkcolor=blue}
\usepackage[initials]{amsrefs}
\graphicspath{{nnls-asymm-views/}}
\numberwithin{equation}{section}
\numberwithin{figure}{section}
\theoremstyle{plain}
\newtheorem{theorem}{Theorem}[section]
\newtheorem{proposition}[theorem]{Proposition}

\theoremstyle{definition}
\newtheorem{assumption}[theorem]{Assumption}
\newtheorem{assumptions}[theorem]{Assumptions}
\newtheorem*{definition*}{Definition}
\newtheorem*{notation*}{Notation}
\newtheorem*{notations*}{Notations}
\newtheorem*{rh-pb*}{Basic RH problem}
\theoremstyle{remark}
\newtheorem{remark}[theorem]{Remark}
\newtheorem*{note*}{Note}

\providecommand{\D}[1]{\mathbb{#1}}
\providecommand{\C}[1]{\mathcal{#1}}
\newcommand{\dd}{\mathrm{d}}
\newcommand{\eul}{\mathrm{e}}
\newcommand{\ii}{\mathrm{i}}
\providecommand{\abs}[1]{\lvert#1\rvert}
\providecommand{\accol}[1]{\lbrace#1\rbrace}

\renewcommand{\Im}{\operatorname{Im}}
\renewcommand{\Re}{\operatorname{Re}}
\newcommand{\err}{\mathrm{err}}
\newcommand{\model}{\mathrm{mod}}
\newcommand{\ord}{\mathrm{O}}

\newcommand{\bc}{\mathrm{BC}}
\newcommand{\const}{\mathrm{const}}
\newcommand{\sol}{\mathrm{sol}}

\title[Focusing NNLS equation with asymmetric boundary conditions]{Focusing nonlocal nonlinear Schr\"odinger equation with asymmetric boundary conditions: large-time behavior}
\author[A.~Boutet de Monvel]{Anne Boutet de Monvel}
\address{AB: Institut de Math\'ematiques de Jussieu-Paris Rive Gauche, Universit\'e de Paris, 8 place Aur\'elie Nemours, 75205 Paris Cedex 13, France}
\email{{\tt anne.boutet-de-monvel@imj-prg.fr}}
\author[Y.~Rybalko]{Yan Rybalko}
\address{YR: B.~Verkin Institute for Low Temperature Physics and Engineering of the National Academy of Sciences of Ukraine, 47 Nauky Avenue, 61103 Kharkiv, Ukraine}
\email{{\tt rybalkoyan@gmail.com}}
\author[D.~Shepelsky]{Dmitry Shepelsky}
\address{DS: B.~Verkin Institute for Low Temperature Physics and Engineering of the National Academy of Sciences of Ukraine, 47 Nauky Avenue, 61103 Kharkiv, Ukraine}
\email{{\tt shepelsky@yahoo.com}}
\begin{document}
\begin{abstract}
We consider the focusing integrable nonlocal nonlinear Schr\"odinger equation
\[
\ii q_{t}(x,t)+q_{xx}(x,t)+2q^{2}(x,t)\bar{q}(-x,t)=0
\]
with asymmetric nonzero boundary conditions: $q(x,t)\to \pm A\eul^{-2\ii A^2t}$ as $x\to\pm\infty$, where $A>0$ is an arbitrary constant. The goal of this work is to study the asymptotics of the solution of the initial value problem for this equation as $t\to+\infty$. For a class of initial values we show that there exist three qualitatively different asymptotic zones in the $(x,t)$ plane. Namely, there are regions where the parameters are modulated (being dependent on the ratio $x/t$) and a central region, where the parameters are unmodulated. This asymptotic picture is reminiscent of that for the defocusing classical nonlinear Schr\"odinger equation, but with some important differences. In particular, the absolute value of the solution in all three regions depends on details of the initial data.
\end{abstract}
\maketitle
\section{Introduction}\label{sec:1}
In the present paper we consider the initial value problem for the focusing nonlocal nonlinear Schr\"odinger (NNLS) equation (we denote a complex conjugate of $q$ by $\bar q$)
\begin{subequations}\label{fasivp}
\begin{alignat}{2}\label{fasivp-a}
& \ii q_{t}(x,t)+q_{xx}(x,t)+2q^{2}(x,t)\bar{q}(-x,t)=0,&\qquad& x\in\D{R},
\quad t\in\D{R},  \\
\label{fasivp-b}
& q(x,0)=q_0(x), &&  x\in\D{R}, 
\end{alignat}
with asymmetric nonzero boundary conditions:
\begin{equation}\label{fasivp-c}
q(x,t) \to \pm A\eul^{-2\ii A^2t},\qquad x\to\pm\infty,\quad t\in\D{R},
\end{equation}
\end{subequations}
for some $A>0$.

\subsubsection*{The NNLS equation}
The integrable NNLS equation was obtained by M. Ablowitz and Z. Musslimani as a nonlocal reduction of the Ablowitz-Kaup-Newell-Segur system \cite{AMP}. This equation satisfies the $\C{PT}$-symmetric condition \cite{BHH}, i.e., $q(x,t)$ and $\overline{q(-x,-t)}$ are its solutions simultaneously. Thus the NNLS equation is related to the non-Hermitian quantum mechanics \cites{BB,EMK18}. Also this equation has connections with the theory of magnetism, because it is gauge equivalent to the complex Landau-Lifshitz equation \cites{GA,R21}. Finally, the NNLS equation is an example of a two-place (Alice-Bob) system \cites{Lou18,LH17}, which involves the values of the solution at not neighboring points, $x$ and $-x$.

The NNLS equation admits exact solutions with distinctive properties. It has both bright and dark soliton solutions \cite{SMMC}, in contrast to its local counterpart, the classical nonlinear Schr\"odinger (NLS) equation. The simplest one-soliton solution of \eqref{fasivp-a} on zero background has, in general, periodic (in time) point singularities \cite{AMP}, so the solution becomes unbounded at these points. Different types of exact solutions with various backgrounds can have such isolated blow-up points in the $(x,t)$ plane. For example, solitons with nonzero boundary conditions \cites{ALM18, AMFL, GP, HL16, LX15, RSs}, rogue waves \cite{YY20} and breathers \cite{San18}. Other important exact solutions of the NNLS equation are given in, e.g., \cites{MS20, MS18, XCLM19}.

\subsubsection*{Initial value problems}
The initial value problem \eqref{fasivp-a}-\eqref{fasivp-b} with nonzero background $q(x,t)\to A\eul^{\ii\theta_\pm(t)}$, as $x\to\pm\infty$ was firstly considered in \cite{ALM18}. It was shown that $\eul^{\ii\theta_\pm(t)}$ remains bounded as $\abs{t}\to\infty$ only in two cases: $\theta_+(t)-\theta_-(t)=0$ or $\theta_+(t)-\theta_-(t)=\pi$. Thus bounded (with respect to $t$) boundary conditions can be either $q(x,t)\to A\eul^{2\ii A^2t}$ as $\abs{x}\to\infty$ or $q(x,t)\to\pm A\eul^{-2\ii A^2t}$ as $x\to\pm\infty$. The inverse scattering transform method for problems with these two boundary values was developed in \cite{ALM18}, where it was shown that the two problems have different continuous spectra. Namely, if $q(x,t)\to A\eul^{2\ii A^2t}$ as $\abs{x}\to\infty$, the continuous spectrum consists of the real line and a vertical band $(-\ii A,\ii A)$, which is reminiscent of the problem for the classical (local) focusing NLS equation on a symmetric \cite{BK14} or step-like \cite{BKS11} background. For $q(x,t)\to\pm A\eul^{-2\ii A^2t}$, $x\to\pm\infty$, the continuous spectrum lies on the real line and has a gap $(-A, A)$, as in the problem for the defocusing NLS equation with symmetric nonzero boundary conditions \cites{DPMV13, IU88, ZS73}. Another interesting feature of problem \eqref{fasivp} is that the boundary functions $\pm A\eul^{-2\ii A^2t}$ are not exact solutions of the NNLS equation. It is in sharp contrast with the local problems, where for the well-posedness it is necessary that the boundary conditions satisfy the equation.

\subsubsection*{Long-time asymptotics}
The long-time asymptotics for the defocusing NLS equation with nonzero boundary conditions manifests important nonlinear phenomena, including solitons \cites{CJ16, V02, ZS73}, rarefaction waves, shock waves, and various plane wave type regions \cites{B89, EGGK, FLQ, IU86, J15}. These developments motivate us to study the asymptotics of problem \eqref{fasivp} and to highlight its qualitative differences with that for the defocusing NLS equation on a nonzero background, which has a similar spectral picture. We also compare the long-time asymptotic behavior of \eqref{fasivp} to that for the Cauchy problem for \eqref{fasivp-a} with boundary conditions $q(x,t)\to A\eul^{2\ii A^2t}$ as $x\to\pm\infty$, which is considered in \cite{RS21PD}.

\subsubsection*{Methods}
The main technical tool used in this paper is the inverse scattering transform method, which allows us to express the solution of \eqref{fasivp} in terms of the solution of an associated Riemann--Hilbert problem. The jump matrix of this problem depends on the parameters $(x,t)$ only via oscillating exponents, so we can apply the Deift and Zhou nonlinear steepest descent method \cites{DZ, DIZ} (see also \cites{DVZ94, DVZ97} for its extensions) to get the asymptotics of the Riemann--Hilbert problem and, therefore, of the solution $q(x,t)$ of \eqref{fasivp}.

\subsubsection*{Organization of the paper}
The article is organized as follows. In Section \ref{fasist} we develop the inverse scattering transform method for \eqref{fasivp} and formulate the basic Riemann--Hilbert problem. We also get the one-soliton solution by using the Riemann--Hilbert approach. Section \ref{faslta} contains our main results, Theorems \ref{fasth1pw} and \ref{fasth2cpr}, on the long-time asymptotic behavior of $q(x,t)$. More precisely, we present the asymptotics in the ``modulated regions'' ($|x/4t|>A/2$) in Theorem \ref{fasth1pw}, and in the central ``unmodulated region'' ($0<|x/4t|<A/2$) in Theorem \ref{fasth2cpr}. Finally, we discuss the transition inside the unmodulated region as $\xi\to 0$. Theorem \ref{fasthtrans} presents the large time asymptotics with $x$ fixed $\neq 0$, in which case $\xi\to 0$.

\section{Inverse scattering transform method}\label{fasist}

The inverse scattering transform formalism for problem \eqref{fasivp} was first developed in \cite{ALM18}. Here we perform the direct and inverse analysis in a different way, in particular we define the inverse transform in terms of an associated Riemann--Hilbert problem formulated in the complex plane of the spectral parameter $k$ entering the standard Lax pair equations for the NNLS equation \eqref{fasivp-a}.
\subsection{Direct scattering}

The NNLS equation \eqref{fasivp-a} is the compatibility condition of the following system of linear equations \cite{AMP} (the ``Lax pair'')
\begin{subequations}\label{fasLP}
\begin{align}
\label{fasLPa}
\Phi_{x}+\ii k\sigma_{3}\Phi&=U\Phi,\\
\label{fasLPb}
\Phi_{t}+2\ii k^{2}\sigma_{3}\Phi&=V\Phi,
\end{align}
\end{subequations}
where $\sigma_3=\left(\begin{smallmatrix} 1& 0\\ 0 & -1\end{smallmatrix}\right)$ is the third Pauli matrix, $\Phi(x,t,k)$ is a $2\times2$ matrix-valued function, $k\in\D{C}$ is the spectral parameter, and $U(x,t)$ and $V(x,t,k)$ are given in terms of $q(x,t)$ as follows:	
\begin{equation}
U(x,t)=\begin{pmatrix}
0& q(x,t)\\
-\bar{q}(-x,t)& 0\\
\end{pmatrix},\qquad 
V(x,t,k)=\begin{pmatrix}
V_{11}(x,t)& V_{12}(x,t,k)\\
V_{21}(x,t,k)& V_{22}(x,t)\\
\end{pmatrix},
\end{equation}
where $V_{11}=-V_{22}=\ii q(x,t)\bar{q}(-x,t)$, $V_{12}=2kq(x,t)+\ii q_x(x,t)$, and $V_{21}=-2k\bar{q}(-x,t)+\ii\bar{q}(-x,t)_x$. 

Assuming that 
\[
\int_{-\infty}^{0}|q(x,t)+A\eul^{-2\ii A^2t}|\,\dd x<\infty\quad\text{and}\quad\int_{0}^{\infty}|q(x,t)-A\eul^{-2\ii A^2t}|\,\dd x<\infty\quad\text{for all }t\geq 0,
\]
we introduce the $2\times2$ matrix valued functions $\Psi_j(x,t,k)$, $j=1,2$ as the solutions of the following linear Volterra integral equations ($j=1,2$)
\begin{align}\label{fasPsi}
\notag
&\Psi_j(x,t,k)=\eul^{-\ii A^2t\sigma_3}\C{E}_j(k)\\
&\qquad+\int\limits_{(-1)^j\infty}^{x}G_j(x,y,t,k)
(U(y,t)-U_j(t))\Psi_j(y,t,k)\eul^{\ii (x-y)f(k)\sigma_3}\,\dd y,
\quad k\in\D{R}\setminus[-A,A].
\end{align}
Here $U_1(t)$ and $U_2(t)$ are the limits of $U(x,t)$ as $x\to\mp\infty$:
\begin{equation}
U(x,t)\to U_j(t),\quad x\to(-1)^j\infty,
\end{equation}
where
\begin{equation}
U_1(t)=
\begin{pmatrix}
0& -A\eul^{-2\ii A^2t}\\
-A\eul^{2\ii A^2t} & 0
\end{pmatrix}
\quad\text{and}\quad
U_2(t)=
\begin{pmatrix}
0& A\eul^{-2\ii A^2t}\\
A\eul^{2\ii A^2t} & 0
\end{pmatrix}.
\end{equation}
The kernels $G_j(x,y,t,k)$, $j=1,2$ are defined in terms of functions $\C{E}_j(k)$, $j=1,2$ and $f(k)$ as follows:
\begin{equation}
G_j(x,y,t,k)=\eul^{-\ii A^2t\sigma_3}\C{E}_j(k)
\eul^{-\ii (x-y)f(k)\sigma_3}\C{E}_j^{-1}(k)
\eul^{\ii A^2t\sigma_3},
\end{equation}
where
\begin{equation}\label{fasK}
\C{E}_j(k)\coloneqq\frac{1}{2}
\begin{pmatrix}
w(k)+\frac{1}{w(k)} &
(-1)^j\,\ii\,\left(w(k)-\frac{1}{w(k)}\right)\\
(-1)^{j+1}\,\ii\,\left(w(k)-\frac{1}{w(k)}\right) & w(k)+\frac{1}{w(k)}
\end{pmatrix},\quad w(k)\coloneqq\left(\frac{k-A}{k+A}\right)^{\frac{1}{4}},
\end{equation}
and
\begin{equation}\label{fasf}
f(k)\coloneqq(k^2-A^2)^{\frac{1}{2}}.
\end{equation}
Here, the functions $f(k)$ and $w(k)$ are defined for $k\in\D{C}\setminus [-A,A]$ as the branches fixed by the large $k$ asymptotics:
\begin{equation}\label{fasbrcut}
f(k)=k+\ord(k^{-1})\quad\text{and}\quad
w(k)=1+\ord(k^{-1}),\quad k\to\infty.
\end{equation}

We denote by $f_\pm(k)$ and $w_\pm(k)$ the limiting values of the corresponding function as $k$ approaches $(-A,A)$ (oriented from $-A$ to $A$) from the left/right side (and similarly for $\C{E}_{j\pm}(k)$). In particular, $f_+(k)=\ii\sqrt{A^2-k^2}$ for $k\in(-A,A)$, with $\sqrt{A^2-k^2}>0$. Observe that $G(x,y,t,k)$ is entire with respect to $k$ for all $x$, $y$, and $t$.

Since $f(k)$ is real for $k\in\D{R}\setminus[-A,A]$, the integral in \eqref{fasPsi} converges for such $k$. Let $Q^{[i]}$ denote the $i$-th column of a matrix $Q$, $\D{C}^{\pm}\coloneqq\accol{k\in\D{C}\mid\pm\Im k>0}$, and $\overline{\D{C}^{\pm}}\coloneqq\accol{k\in\D{C}\mid\pm\Im k\geq 0}$. Then we can define $\Psi_j^{[j]}(x,t,k)$, $j=1,2$, and $\Psi_1^{[2]}(x,t,k)$, $\Psi_2^{[1]}(x,t,k)$ on the cut $(-A, A)$ as the limiting values from $\D{C}^+$ and $\D{C}^-$, respectively:
\begin{align}\label{fasPsiC+}
\notag
&\Psi_{j+}^{[j]}(x,t,k)
=\eul^{-\ii A^2t\sigma_3}\C{E}_{j+}^{[j]}(k)\\
&+\int\limits_{(-1)^j\infty}^{x}G_j(x,y,t,k)
(U(y,t)-U_j(t))\Psi_{j+}^{[j]}(y,t,k)
\eul^{(-1)^{j+1}\ii (x-y)f_+(k)}\,\dd y,\quad k\in(-A,A),
\end{align}
and
\begin{subequations}\label{fasPsiC-}
\begin{align}
\notag
&\Psi_{1-}^{[2]}(x,t,k)
	=\eul^{-\ii A^2t\sigma_3}\C{E}_{1-}^{[2]}(k)\\
&\qquad+\int\limits_{-\infty}^{x}G_1(x,y,t,k)(U(y,t)-U_1(t))
	\Psi_{1-}^{[2]}(y,t,k)\eul^{-\ii (x-y)f_-(k)}\,\dd y,\quad k\in(-A,A),\\
\notag
&\Psi_{2-}^{[1]}(x,t,k)=\eul^{-\ii A^2t\sigma_3}\C{E}_{2-}^{[1]}(k)\\
&\qquad+\int\limits_{+\infty}^{x}G_2(x,y,t,k)(U(y,t)-U_2(t))
\Psi_{2-}^{[1]}(y,t,k)\eul^{\ii (x-y)f_-(k)}\,\dd y,\quad k\in(-A,A).
\end{align}
\end{subequations}
Moreover, when the solution $q(x,t)$ converges exponentially fast to its boundary values, we can define $\Psi_{j-}^{[j]}(x,t,k)$, $j=1,2$, and $\Psi_{1+}^{[2]}(x,t,k)$, $\Psi_{2+}^{[1]}(x,t,k)$ for $k\in(-A,A)$ by integral equations similar to \eqref{fasPsiC+} and \eqref{fasPsiC-}, respectively.

\begin{proposition}[properties of $\Psi_j$]\label{fasproppsi1}
$\Psi_1(x,t,k)$ and $\Psi_2(x,t,k)$ have the following properties.

\emph{(i)}
The columns $\Psi_1^{[1]}(x,t,k)$ and $\Psi_2^{[2]}(x,t,k)$ are analytic for $k\in\D{C}^+$ and continuous for $k\in\overline{\D{C}^+}\setminus\{\pm A\}$, where $\Psi_j^{[j]}(x,t,k)$ is identified with $\Psi_{j+}^{[j]}(x,t,k)$, $j=1,2$ for $k\in(-A,A)$.

$\Psi_1^{[1]}(x,t,k)$ and $\Psi_2^{[2]}(x,t,k)$ have the following behaviors at $k=\infty$ and $k=\pm A$:
\begin{alignat*}{4}
&\Psi_1^{[1]}(x,t,k)=\eul^{-\ii A^2t}
\begin{pmatrix}
1\\
0\end{pmatrix}
+\ord(k^{-1}),&\quad&\Psi_2^{[2]}(x,t,k)=\eul^{\ii A^2t}
\begin{pmatrix}
0\\
1\end{pmatrix}
+\ord(k^{-1}),&\quad&k\to\infty,&\quad&k\in\D{C}^+,\\
&\Psi_1^{[1]}(x,t,k)=\ord\bigl((k\mp A)^{-\frac{1}{4}}\bigr),
&&\Psi_2^{[2]}(x,t,k)=\ord\bigl((k\mp A)^{-\frac{1}{4}}\bigr),&&k\to\pm A,&&k\in\D{C}^+.
\end{alignat*}

\emph{(ii)}
The columns $\Psi_1^{[2]}(x,t,k)$ and $\Psi_2^{[1]}(x,t,k)$ are analytic for $k\in\D{C}^-$ and continuous for $k\in\overline{\D{C}^-}\setminus\{\pm A\}$, where $\Psi_1^{[2]}(x,t,k)$ and $\Psi_2^{[1]}(x,t,k)$ are identified with $\Psi_{1-}^{[2]}(x,t,k)$ and $\Psi_{2-}^{[1]}(x,t,k)$ for $k\in(-A,A)$.

$\Psi_1^{[2]}(x,t,k)$ and $\Psi_2^{[1]}(x,t,k)$ have the following behaviors at $k=\infty$ and $k=\pm A$:
\begin{alignat*}{4}
&\Psi_1^{[2]}(x,t,k)=\eul^{\ii A^2t}
\begin{pmatrix}
0\\
1\end{pmatrix}
+\ord(k^{-1}),&\quad&\Psi_2^{[1]}(x,t,k)=\eul^{-\ii A^2t}
\begin{pmatrix}
1\\
0\end{pmatrix}
+\ord(k^{-1}),&\quad&k\to\infty,&\quad&k\in\D{C}^-,\\
&\Psi_1^{[2]}(x,t,k)=\ord\bigl((k\mp A)^{-\frac{1}{4}}\bigr),&&\Psi_2^{[1]}(x,t,k)=\ord\bigl((k\mp A)^{-\frac{1}{4}}\bigr),&&k\to\pm A,&&k\in\D{C}^-.
\end{alignat*}

\emph{(iii)}
The functions $\Phi_j(x,t,k)$, $j=1,2$ defined by
\begin{align}\label{fsjost}
\Phi_j(x,t,k)\coloneqq\Psi_j(x,t,k)\eul^{-(\ii x+2\ii tk)f(k)\sigma_3},\quad k\in\D{R}\setminus [-A,A],
\end{align}
are the (Jost) solutions of the Lax pair \eqref{fasLP} satisfying the boundary conditions
\begin{align}
\Phi_j(x,t,k)-\Phi_j^{\bc}(x,t,k)\to 0,\quad x\to (-1)^j\infty,\quad k\in\D{R}\setminus [-A,A],
\end{align}
where $\Phi_j^{\bc}(x,t,k)\coloneqq\eul^{-\ii A^2t\sigma_3}
\C{E}_j(k)\eul^{-(\ii x+2\ii tk)f(k)\sigma_3}$.
		
\emph{(iv)}
$\det\Psi_j(x,t,k)\equiv 1$ for $k\in\D{R}\setminus[-A,A]$.
		
\emph{(v)}
The  following symmetry relations hold:
\begin{subequations}\label{fssymmpsi}
\begin{equation}\label{fassymmR}
\begin{split}
\sigma_1\overline{\Psi_1^{[1]}(-x,t,-\bar{k})}&=
\Psi_2^{[2]}(x,t,k),\quad k\in\overline{\D{C}^+}\setminus[-A,A],\\
\sigma_1\overline{\Psi_{1+}^{[1]}(-x,t,-k)}&=
\Psi_{2+}^{[2]}(x,t,k),\quad k\in(-A,A),\\
\sigma_1\overline{\Psi_1^{[2]}(-x,t,-\bar{k})}&=
\Psi_2^{[1]}(x,t,k),\quad k\in\overline{\D{C}^-}\setminus[-A,A],\\
\sigma_1\overline{\Psi_{1-}^{[2]}(-x,t,-k)}&=
\Psi_{2-}^{[1]}(x,t,k),\quad k\in (-A,A),
\end{split}
\end{equation}
and 
\begin{equation}\label{fassymmseg}
\Psi_{1+}^{[1]}(x,t,k)=-\Psi_{1-}^{[2]}(x,t,k),\quad
\Psi_{2+}^{[2]}(x,t,k)=-\Psi_{2-}^{[1]}(x,t,k),\quad k\in(-A,A),
\end{equation}
\end{subequations}
where $\sigma_1=\bigl(\begin{smallmatrix}0& 1\\1 & 0\end{smallmatrix}\bigl)$ is the first Pauli matrix.

Moreover, when $\Psi_{j-}^{[j]}(x,t,k)$, $j=1,2$ and $\Psi_{1+}^{[2]}(x,t,k)$, $\Psi_{2+}^{[1]}(x,t,k)$ exist (e.g., when $q(x,t)$ converges exponentially fast to its boundary values), they satisfy the following conditions:
\begin{equation}
\label{fassymmsegadd}
\Psi_{1-}^{[1]}(x,t,k)=\Psi_{1+}^{[2]}(x,t,k),\quad
\Psi_{2-}^{[2]}(x,t,k)=\Psi_{2+}^{[1]}(x,t,k),\quad k\in(-A,A).
\end{equation}
\end{proposition}

\begin{proof}
Items (i)--(iii) follow directly from the integral equations \eqref{fasPsi}. Since the matrix $U(x,t)$ is traceless and $\det\C{E}_j(k)=1$, $j=1,2$, we get item (iv). Finally, \eqref{fassymmR} in item (v) follows from the symmetries 
\begin{equation}
\sigma_1\overline{U}(-x,t)\sigma_1^{-1}=-U(x,t),\quad\text{and}\quad 
\sigma_1\overline{G_1(-x,-y,t,-\bar{k})}\sigma_1^{-1}=G_2(x,y,t,k),\quad k\in\D{C}, 
\end{equation}
whereas \eqref{fassymmseg} and \eqref{fassymmsegadd} follow from the symmetries
\begin{equation}\label{fassymE12}
\C{E}_{j+}(k)=(-1)^{j+1}\,\ii \,\C{E}_{j-}(k)\sigma_2,\quad j=1,2,\quad k\in (-A,A),
\end{equation}
where $\sigma_2=\bigl(\begin{smallmatrix}0& -\ii \\ \ii & 0\end{smallmatrix}\bigl)$ is the second  Pauli matrix.
\end{proof}

\subsection{Spectral functions}\label{fassectspfunct}
The Jost solutions $\Phi_1(x,t,k)$ and $\Phi_2(x,t,k)$ of the Lax pair \eqref{fasLP} are related by a matrix independent of $x$ and $t$, which allows us to introduce the  scattering matrix $S(k)$ as follows:
\begin{equation}\label{fssr}
\Phi_1(x,t,k)=\Phi_2(x,t,k)S(k),\quad k\in\D{R}\setminus [-A,A],
\end{equation}
or, in terms of $\Psi_j(x,t,k)$, $j=1,2$,
\begin{equation}\label{fassrpsi}
\Psi_1(x,t,k)=\Psi_2(x,t,k)\eul^{-(\ii x+2\ii tk)f(k)\sigma_3}S(k)\eul^{(\ii x+2\ii tk)f(k)\sigma_3},\quad k\in\D{R}\setminus [-A,A].
\end{equation}
From the symmetry relations \eqref{fassymmR} it follows that  $S(k)$ can be written as
\begin{equation}\label{fsscatt}
S(k)=\begin{pmatrix}
a_1(k) & -\overline{b(-k)}\\
b(k) & a_2(k)
\end{pmatrix},\quad k\in\D{R}\setminus[-A,A].
\end{equation}
Note that due to the Schwarz symmetry breaking for the solutions $\Psi_j(x,t,k)$, $j=1,2$, see \eqref{fassymmR},  the values of $a_1(k)$ for $k\in\D{C}^+$ and $a_2(k)$ for $k\in\D{C}^-$ are, in general, \emph{not} related. In particular, this implies that $a_1(k)$ and $a_2(k)$ can have different numbers of zeros in the corresponding complex half-planes.

Relation \eqref{fassrpsi} implies that $a_1(k)$, $a_2(k)$, and $b(k)$ can be found in terms of the initial data alone via the following determinants:
\begin{subequations}\label{fasa1a2b}
\begin{alignat}{2}
a_1(k)&=\det\bigl(\Psi_1^{[1]}(0,0,k),\Psi_2^{[2]}(0,0,k)\bigr),&\qquad&k\in\overline{\D{C}^{+}}\setminus[-A,A],\\
a_2(k)&=\det\bigl(\Psi_2^{[1]}(0,0,k),\Psi_1^{[2]}(0,0,k)\bigr),&&k\in\overline{\D{C}^-}\setminus[-A,A],\\
b(k)&=\det\bigl(\Psi_2^{[1]}(0,0,k),\Psi_1^{[1]}(0,0,k)\bigr),&&k\in\D{R}\setminus [-A,A].
\end{alignat}
\end{subequations}

From \eqref{fasa1a2b} and Proposition \ref{fasproppsi1} (i) and (ii) we conclude that $a_j(k)$, $j=1,2$, and $b(k)$ have the following large $k$ behaviors:
\begin{alignat*}{3}
a_1(k)&=1+\ord(k^{-1}),&\quad&k\in \overline{\D{C}^{+}},&\quad&k\to\infty,\\
a_2(k)&=1+\ord(k^{-1}),&&k\in \overline{\D{C}^-},&&k\to\infty,\\
b(k)&=\ord(k^{-1}),&&k\in\D{R},&&k\to\infty.
\end{alignat*}

Defining $a_{1+}(k)$ and $a_{2-}(k)$ for $k\in(-A,A)$ as the limits of $a_{1}(k)$ and $a_{2}(k)$ from $\D{C}^{+}$ and $\D{C}^-$ respectively, we have
\begin{equation}
\begin{split}
a_{1+}(k)&=\det\bigl(\Psi_{1+}^{[1]}(0,0,k),
\Psi_{2+}^{[2]}(0,0,k)\bigr),\quad k\in (-A,A),\\
a_{2-}(k)&=\det\bigl(\Psi_{2-}^{[1]}(0,0,k),
\Psi_{1-}^{[2]}(0,0,k)\bigr),
\quad k\in (-A,A).
\end{split}
\end{equation}
Moreover, when the initial data $q_0(x)$ converges exponentially fast to its boundary values, we can define $a_{1-}(k)$, $a_{2+}(k)$ and $b_{\pm}(k)$ for $k\in(-A,A)$ by taking the corresponding limits in \eqref{fasa1a2b}:
\begin{subequations}
\begin{alignat}{2}
a_{1-}(k)&=\det\bigl(\Psi_{1-}^{[1]}(0,0,k),\Psi_{2-}^{[2]}(0,0,k)\bigr),&\quad& k\in (-A,A),\\
a_{2+}(k)&=\det\bigl(\Psi_{2+}^{[1]}(0,0,k),\Psi_{1+}^{[2]}(0,0,k)\bigr),&&k\in (-A,A),\\
b_\pm(k)&=\det\bigl(\Psi_{2\pm}^{[1]}(0,0,k),\Psi_{1\pm}^{[1]}(0,0,k)\bigr),&&k\in (-A,A).
\end{alignat}
\end{subequations}
The symmetry relations \eqref{fssymmpsi} yield the following symmetries of the spectral functions:
\begin{equation}\label{fasajsym}
\overline{a_1(-\bar{k})}=a_1(k),\quad k\in\overline{\D{C}^{+}}\setminus[-A, A]\quad\text{and}\quad
\overline{a_2(-\bar{k})}=a_2(k),\quad k\in\overline{\D{C}^-}\setminus[-A, A],
\end{equation}
whereas \eqref{fassymmsegadd} implies that 
\begin{equation}\label{fasajsymA}
a_{1\pm}(k)=-a_{2\mp}(k)\quad\text{and}\quad b_{\pm}(k)=-\overline{b_{\mp}(-k)},\quad k\in (-A,A).
\end{equation}
From Proposition \ref{fasproppsi1} (iv), \eqref{fsjost}, and \eqref{fssr} it follows that $a_1(k)$, $a_2(k)$, and $b(k)$ satisfy the determinant relations:
\begin{equation}\label{fasdeterm}
\begin{split}
a_1(k)a_2(k)+b(k)\overline{b(-k)}&=1,\quad k\in\D{R}\setminus [-A,A],\\
a_{1\pm}(k)a_{2\pm}(k)+b_{\pm}(k)\overline{b_{\pm}(-k)}&=1,\quad  k\in (-A,A).
\end{split}
\end{equation}
Finally, we point out that $a_1(k)$, $a_2(k)$, and $b(k)$ are $\ord\bigl((k\mp A)^{-\frac{1}{2}}\bigr)$ as $k\to\pm A$.

\begin{proposition}[pure step initial data]\label{fasbgiv}
Consider problem \eqref{fasivp} with initial data
\begin{equation}\label{fasstepiv}
q_0(x)=q_{0,R}(x)=
\begin{cases}
A, &x>R,\\
-A, &x<R,
\end{cases}
\end{equation}
for some $A>0$ and $R\in\D{R}$. Introduce 
\begin{equation}\label{fash}
 h(k)\coloneqq(k^2+A^2)^{\frac{1}{2}},
\end{equation}
which is defined in $\D{C}\setminus[-\ii A,\ii A]$ and is fixed by the 
asymptotics $h(k)=k+\ord(k^{-1})$ as $k\to\infty$. Define 
\begin{equation}\label{faslambda}
\lambda_j(k)\coloneqq\ii(f(k)+(-1)^{j+1}h(k)),\quad j=1,2. 
\end{equation}
Then the spectral functions associated with this problem have the following form, according to the sign of $R\in\D{R}$:
\begin{enumerate}[\rm(i)]
\item 
For $R>0$,
\begin{subequations}\label{fasR>0}
\begin{align}\label{fasR>0a_1}
a_1(k)&=\frac{1}{2f(k)h(k)}
	\bigl(\eul^{2\lambda_1(k)R}
	\bigl(A^2+\ii k\lambda_2(k)\bigr)
	-\eul^{2\lambda_2(k)R}
	\bigl(A^2+\ii k\lambda_1(k)\bigr)
	\bigr),\\
\label{fasR>0a_2}
a_2(k)&=\frac{1}{2f(k)h(k)}
	\bigl(\eul^{-2\lambda_2(k)R}
	\bigl(A^2-\ii k\lambda_1(k)\bigr)
	-\eul^{-2\lambda_1(k)R}
	\bigl(A^2-\ii k\lambda_2(k)\bigr)
	\bigr),\\
\label{fasR>0b}
b(k)&=\frac{-\ii A}{2f(k)h(k)}
	\bigl(\eul^{2\ii h(k)R}\bigl(h(k)+k\bigr)
	+\eul^{-2\ii h(k)R}\bigl(h(k)-k\bigr)
	\bigr).
\end{align}
\end{subequations}
\item 
For $R=0$,
\begin{equation}\label{fasR=0}
a_1(k)=a_2(k)=\frac{k}{f(k)},\qquad b(k)=\frac{-\ii A}{f(k)}.
\end{equation}
\item 
For $R<0$,
\begin{subequations}\label{fasR<0}
\begin{align}\label{fasR<0a_1}
a_1(k)&=\frac{1}{2f(k)h(k)}
	\bigl(\eul^{-2\lambda_2(k)R}
	\bigl(A^2-\ii k\lambda_1(k)\bigr)
	-\eul^{-2\lambda_1(k)R}
	\bigl(A^2-\ii k\lambda_2(k)\bigr)
	\bigr),\\
\label{fasR<0a_2}
a_2(k)&=\frac{1}{2f(k)h(k)}
	\bigl(\eul^{2\lambda_1(k)R}
	\bigl(A^2+\ii k\lambda_2(k)\bigr)
	-\eul^{2\lambda_2(k)R}
	\bigl(A^2+\ii k\lambda_1(k)\bigr)
	\bigr),\\
\label{fasR<0b}
b(k)&=\frac{-\ii A}{2f(k)h(k)}
	\bigl(\eul^{2\ii h(k)R}
	\bigl(h(k)+k\bigr)
	+\eul^{-2\ii h(k)R}
	\bigl(h(k)-k\bigr)
	\bigr).
\end{align}
\end{subequations}
\end{enumerate}
\end{proposition}

\begin{proof}
See Appendix \ref{appendix}.
\end{proof}

\begin{remark}
Note that for any $R\in\D{R}$, $a_1(k)$, $a_2(k)$, and $b(k)$ have no jump across $[-\ii A,\ii A]$. Also, if we take the limits $R\to\pm0$ in the expressions of the spectral functions for $R>0$ and $R<0$, we arrive at \eqref{fasR=0}.
\end{remark}

\begin{remark}
The NNLS equation is not translation invariant. Therefore, shifting the initial data by a constant value can drastically affect the behavior of the solution \cite{RS21CIMP}. Formulas \eqref{fasR>0}--\eqref{fasR<0} illustrate this in terms of the spectral functions in the case of pure step initial data \eqref{fasstepiv}.
\end{remark}

The scattering map associates to $q_0(x)$ 
\begin{enumerate}[(i)]
\item
the spectral functions $b(k)$ and $a_j(k)$, $j=1,2$,
\item
the discrete data, which are the zeros of $a_j(k)$, $j=1,2$ and the associated norming constants.
\end{enumerate}
In studying initial value problems for integrable nonlinear PDEs, the assumptions about these zeros usually rely on properties of the discrete spectrum associated with step-like initial data involving prescribed boundary values, like \eqref{fasivp-c} (see, e.g., \cites{BKS11, BM17, IU86, J15, RS21CIMP}). Alternatively, the discrete spectrum can be added to the formulation of the associated Riemann--Hilbert problem for studying the evolution of more general initial data, which includes solitons \cites{CJ16, V02, ZS73}.

In the present paper we consider initial data which are characterized in spectral terms and which are motivated by the pure step initial data with $R=0$. Namely, we make the following assumptions.

\begin{assumptions}[on the zeros of the spectral functions $a_1(k)$ and $a_2(k)$]\label{assumpts}
We assume that 
\begin{enumerate}[({A}1)]
\item 
$a_1(k)$ and $a_2(k)$ do not have zeros in $\overline{\D{C}^+}\setminus(-A,A)$ and $\overline{\D{C}^-}\setminus(-A,A)$, respectively;
\item
for $k\in (-A,A)$, both $a_{1+}(k)$ and $a_{2-}(k)$ have a simple zero at $k=0$, i.e.,
\begin{equation}\label{fask=0}
\begin{split}
a_{1+}(k)&=a_{10}k+\ord(k^2),\quad k\to 0,\quad a_{10}\neq 0,\\
a_{2-}(k)&=a_{20}k+\ord(k^2),\quad k\to 0,\quad a_{20}\neq 0.
\end{split}
\end{equation}
\end{enumerate}
Then from \eqref{fasajsymA} and \eqref{fasajsym} it follows that
\begin{equation}
a_{20}=-a_{10}\quad\text{and}\quad\Re a_{10}=0.
\end{equation}
\end{assumptions}

\subsection{Riemann--Hilbert problem}
Taking into account the analytical properties of the columns of the matrices $\Psi_j(x,t,k)$, $j=1,2$ (see Proposition \ref{fasproppsi1} (i) and (ii)), we define the $2\times2$ sectionally holomorphic matrix $M(x,t,k)$ as follows:
\begin{equation}
\label{fasM}
M(x,t,k)=
\begin{cases}
\eul^{\ii A^2t\sigma_3}\Bigl(\frac{\Psi_1^{[1]}(x,t,k)}{a_{1}(k)},\Psi_2^{[2]}(x,t,k)\Bigr),&k\in\D{C}^+,\\
\eul^{\ii A^2t\sigma_3}\Bigl(\Psi_2^{[1]}(x,t,k),\frac{\Psi_1^{[2]}(x,t,k)}{a_{2}(k)}\Bigr),&k\in\D{C}^-.
\end{cases}
\end{equation}
By Assumptions \ref{assumpts}, $a_1(k)$ and $a_2(k)$ have no zeros in the corresponding half-planes and thus the matrix $M(x,t,k)$ does not have poles in $\D{C}\setminus\D{R}$. From the scattering relation \eqref{fassrpsi}, the symmetries \eqref{fassymmseg}, and the relations \eqref{fasajsymA} it follows that $M(x,t,k)$ satisfies a multiplicative jump condition:
\begin{subequations}\label{fasjump}
\begin{equation}\label{fasj}
M_+(x,t,k)=M_-(x,t,k)J(x,t,k),\quad k\in\D{R}.
\end{equation}
Here and below $M_{+}(\,\cdot\,,\,\cdot\,,k)$ and $M_-(\,\cdot\,,\,\cdot\,,k)$ denote the nontangental limits of $M(\,\cdot\,,\,\cdot\,,k)$ as $k$ approaches the contour from the left and right sides, respectively (here, the real line $\D{R}$ is oriented from $-\infty$ to $+\infty$). The jump matrix $J(x,t,k)$ has the following form:
\begin{equation}\label{fasjbrh}
J(x,t,k)=
\begin{cases}
\begin{pmatrix}
1+r_1(k)r_2(k) & r_2(k)\eul^{-(2\ii x+4\ii tk)f(k)}\\
r_1(k)\eul^{(2\ii x+4\ii tk)f(k)} & 1
\end{pmatrix},&k\in\D{R}\setminus [-A,A],\\
-\ii\sigma_2,&k\in(-A,A),
\end{cases}
\end{equation}
with the reflection coefficients
\begin{equation}\label{fasr_j}
r_1(k)\coloneqq\frac{b(k)}{a_1(k)}\quad\text{and}\quad
r_2(k)\coloneqq\frac{\overline{b(-k)}}{a_2(k)},\quad k\in\D{R}\setminus[-A, A].
\end{equation}
\end{subequations}

\begin{remark}\label{faszerr1r2}
If $b(k)$ can be analytically continued into a band containing $\D{R}$, we can also define $r_j(k)$ in this band. Then in view of \eqref{fasajsymA}, $r_{1\pm}(k)=r_{2\mp}(k)$ and therefore $1+r_1(k)r_2(k)$ does not have a jump across $(-A, A)$. From the determinant relation \eqref{fasdeterm} it follows that $1+r_1(k)r_2(k)=a_1^{-1}(k)a_2^{-1}(k)$, so $1+r_1(k)r_2(k)$ can have simple zeros at $k=\pm A$. This takes place, e.g., for pure step initial data \eqref{fasstepiv} (see~\cite{J15}*{Section 3}).
\end{remark}

In view of Proposition \ref{fasproppsi1} (i) and (ii), and Assumptions \ref{assumpts}, $M(x,t,k)$ has weak singularities at $k=\pm A$:
\begin{equation}\label{faspmA}
M(x,t,k)=\ord\bigl((k\pm A)^{-\frac{1}{4}}\bigr),\quad k\to\mp A.
\end{equation}
Also it has the normalization condition for large $k$:
\begin{equation}\label{faskinf}
M(x,t,k)=I+\ord(k^{-1}),\quad k\to\infty,
\end{equation}
where $I$ is the identity matrix. Finally, $M(x,t,k)$ satisfies the following conditions at $k=0$:
\begin{subequations}\label{fask=0bRH}
\begin{align}
&\lim_{\substack{k\to 0,\\k\in\D{C}^+}}
	kM^{[1]}(x,t,k)=\frac{\gamma_+}{a_{10}}\eul^{-2Ax}
	M^{[2]}_+(x,t,0),\\
&\lim_{\substack{k\to 0,\\k\in\D{C}^-}}
	kM^{[2]}(x,t,k)=\frac{\gamma_-}{a_{20}}\eul^{-2Ax}
	M^{[1]}_-(x,t,0),
\end{align}
\end{subequations}
where $a_{10}$ and $a_{20}$ were introduced in \eqref{fask=0}, and $\gamma_\pm$ are defined as follows:
\[
\Phi_{1+}^{[1]}(x,t,0)=\gamma_+\Phi_{2+}^{[2]}(x,t,0)\quad\text{and}\quad\Phi_{1-}^{[2]}(x,t,0)=\gamma_-\Phi_{2-}^{[1]}(x,t,0).
\]
From \eqref{fassymmseg} and \eqref{fassymmR} one concludes that $\gamma_+=\gamma_-$ and $|\gamma_+|=1$.

\begin{remark} \label{fasbgamma}
If $b(k)$ can be analytically continued into a band, the norming constants $\gamma_\pm$ can be found in terms of $b(k)$ as follows: $\gamma_+=b_+(0)$ and $\gamma_-=-\overline{b_-(0)}$.
\end{remark}

Thus we arrive at the following basic Riemann--Hilbert (RH) problem: 

\begin{rh-pb*}
Find a sectionally analytic $2\times 2$ matrix $M(x,t,k)$, which \begin{enumerate}[(i)]
\item
satisfies the jump condition \eqref{fasjump} across the real axis, 
\item
has weak singularities \eqref{faspmA} at $k=\pm A$, 
\item
converges to the identity matrix as $k\to\infty$, 
\item
and satisfies the singularity conditions \eqref{fask=0bRH} at $k=0$.
\end{enumerate}
\end{rh-pb*}

Using standard arguments based on Liouville's theorem, it can be shown that the solution of this RH problem is unique, if it exists.

The solution $q(x,t)$ of the initial value problem \eqref{fasivp} can be found from the large $k$ expansion of the solution $M(x,t,k)$ of the basic RH problem (follows from \eqref{fasLPa}):
\begin{equation}\label{fasasol}
q(x,t)=2\ii \eul^{-2\ii A^2t}
\lim_{k\to\infty}kM_{12}(x,t,k),\quad
q(-x,t)=-2\ii \eul^{-2\ii A^2t}
\lim_{k\to\infty}\overline{kM_{21}(x,t,k)}.
\end{equation}
Thus both $q(x,t)$ and $q(-x,t)$ can be found from $M(x,t,k)$ evaluated for $x\geq0$.

\begin{remark}
Since the jump matrix $J(x,t,k)$ satisfies the condition
\begin{equation}
	\sigma_1 \overline{J(-x,t,-k)} \sigma_1^{-1} = \begin{pmatrix}
	a_2(k) & 0 \\ 0 & \frac{1}{a_2(k)}\end{pmatrix} J(x,t,k)\begin{pmatrix}
	a_1(k) & 0 \\ 0 & \frac{1}{a_1(k)}
	\end{pmatrix},
	\quad k\in\D{R}\setminus\{\pm A\},
\end{equation}
the solution $M(x,t,k)$ of the basic RH problem satisfies the following symmetry conditions (see \cite{RSs}*{(2.55)}):
\begin{equation}
	M(x,t,k)=\begin{cases}
	\sigma_1 \overline{M(-x,t,-\bar k)}\sigma_1^{-1} 
	\begin{pmatrix}
	\frac{1}{a_1(k)} & 0 \\ 0 & a_1(k)
	\end{pmatrix}, & k\in{\mathbb C}^+, \\
	\sigma_1 \overline{M(-x,t,-\bar k)}\sigma_1^{-1}
	\begin{pmatrix}
	a_2(k) & 0 \\ 0 & \frac{1}{a_2(k)}
	\end{pmatrix}, & k\in{\mathbb C}^-.
	\end{cases}
\end{equation}
\end{remark}

\subsection{One-soliton solution}\label{fasonesol}
The one-soliton solution of the focusing NNLS equation satisfying boundary conditions \eqref{fasivp-c} was obtained in \cite{HL16}*{Section 4}, by using the Darboux transformation and in \cite{ALM18}*{Section 3} via the inverse scattering transform method. Here we rederive this soliton solution using the Riemann--Hilbert approach. Consider the basic RH problem in the reflectionless case, i.e., with $r_1(k)\equiv r_2(k)\equiv0$:
\begin{subequations}\label{fasMsol}
\begin{alignat}{2}\label{fasMsoljump}
M^{\sol}_+(x,t,k)&=-\ii M^{\sol}_-(x,t,k)\sigma_2,&\quad&
k\in (-A,A),\\
M^{\sol}(x,t,k)&=I+\ord(k^{-1}), &&k\to\infty,\\
M^{\sol}(x,t,k)&=\ord\bigl((k\mp A)^{-\frac{1}{4}}\bigr),&& k\to\pm A
\end{alignat}
\end{subequations}
and with conditions at $k=0$ of type 
\eqref{fask=0bRH}:
\begin{subequations}\label{fask=0sol}
\begin{align}
\label{fask=0sola}
&\lim_{\substack{k\to 0,\\k\in\D{C}^+}}k(M^{\sol})^{[1]}(x,t,k)=d_0\, \eul^{-2Ax}(M^{\sol})^{[2]}_+(x,t,0),\\
&\lim_{\substack{k\to 0,\\k\in\D{C}^-}}k(M^{\sol})^{[2]}(x,t,k)=-d_0\,\eul^{-2Ax}(M^{\sol})^{[1]}_-(x,t,0),
\end{align}
\end{subequations}
for some $d_0=\frac{\gamma_+}{a_{10}}$, with $|\gamma_+|=1$.

In the reflectionless case, the spectral functions $a_1(k)$ and $a_2(k)$ are as follows (see the trace formula in \cite{ALM18}*{Section 3}):
\begin{equation}\label{fasa_1a_2_sol}
a_1(k)=\frac{k+f(k)-\ii A}{k+f(k)+\ii A}\quad\text{and}\quad
a_2(k)=\frac{k+f(k)+\ii A}{k+f(k)-\ii A}.
\end{equation}
From \eqref{fasa_1a_2_sol} we have $a_{10}=-\frac{\ii}{2A}$ (see \eqref{fask=0}), which implies that
\begin{equation}\label{fasd_0mod}
d_0=2A\eul^{\ii\phi_0}\quad\text{with some }\phi_0\in\D{R}.
\end{equation}
The jump and singularity conditions \eqref{fasMsoljump} and \eqref{fask=0sol} imply that the solution of the RH problem above can be written in the form
\begin{equation}
\label{fasMssol}
M^{\sol}(x,t,k)=N(x,t,k)\C{E}_2(k),\quad k\in\D{C}\setminus\{\pm A, 0\},
\end{equation}
where $\C{E}_2(k)$ is defined in \eqref{fasK} and $N(x,t,k)=I+\frac{N_1(x,t)}{k}$ with some matrix $N_1(x,t)$. On the other hand, conditions \eqref{fask=0sol} imply that $M_+(x,t,k)$ can be written as follows:
\begin{equation}
M_+(x,t,k)=
\begin{pmatrix}
\alpha(x,t)& 0\\
0& \beta(x,t)
\end{pmatrix}
\left(
\begin{pmatrix}
d_0\eul^{-2Ax}& 1\\
d_0\eul^{-2Ax}& 1
\end{pmatrix}
+P(x,t)k+\ord(k^2)
\right)
\begin{pmatrix}
1/k& 0\\
0& 1
\end{pmatrix},
\quad k\to 0,
\end{equation}
with some scalars $\alpha(x,t)$, $\beta(x,t)$, and a matrix-valued function $P(x,t)$. Then, using the relation $N(x,t,k)=M^{\sol}(x,t,k)\C{E}^{-1}_2(k)$ and
\begin{equation}
\C{E}^{-1}_{2+}(k)=
\frac{1}{\sqrt{2}}
\begin{pmatrix}
1 & 1\\
-1 & 1
\end{pmatrix}
+\frac{\ii k}{2\sqrt{2}A}
\begin{pmatrix}
-1& 1\\
-1& -1
\end{pmatrix}
+\ord(k^2),\quad k\to 0,
\end{equation}
we conclude that $\alpha$, $\beta$, and $N_1$ are independent of $t$. Moreover,
\begin{equation}
N_1(x)=d_0\eul^{-2Ax}
\begin{pmatrix}
\alpha(x)& 0\\
\beta(x)& 0
\end{pmatrix}
\C{E}_{2+}^{-1}(0)
\quad\text{with}\quad
\alpha(x)= - \beta(x) = -\frac{\sqrt{2}A}{2A+\ii d_0\eul^{-2Ax}}.
\end{equation}
Thus $M^{\sol}(x,t,k)$ is independent of $t$ and has the form
\begin{equation}\label{fasMssol-1}
M^{\sol}(x,t,k)=
\left(
I+\frac{\mu(x)}{k}
\begin{pmatrix}
-1& -1\\
1& 1
\end{pmatrix}
\right)
\C{E}_2(k)
\end{equation}
with $\mu(x)=\frac{Ad_0\eul^{-2Ax}}{2A+\ii d_0\eul^{-2Ax}}$. Finally, using \eqref{fasasol} and the notation $\phi_0$ from \eqref{fasd_0mod}, we obtain the exact one-soliton solution as follows (see \cite{ALM18}*{(3.106)} and \cite{HL16}*{(17)}):
\begin{equation}\label{fassoliton}
q(x,t)=A\eul^{-2\ii A^2t}\left(1-\frac{2\ii \eul^{-2Ax+\ii \phi_0}}{1+\ii \eul^{-2Ax+\ii\phi_0}}\right)\equiv A\eul^{-2\ii A^2t}\tanh(Ax-\ii \phi_0/2-\ii \pi/4).
\end{equation}

\section{Long-time asymptotic analysis}\label{faslta}
\subsection{Signature table}
Introduce the phase function $\theta(k,\xi)$ as follows:
\begin{equation}
\theta(k,\xi)\coloneqq 4\xi f(k)+2kf(k),\quad \xi\coloneqq\frac{x}{4t}.
\end{equation}
As noticed above, we can consider $\xi\ge 0$ only. In terms of $\theta(k,\xi)$, the exponentials in \eqref{fasjbrh}  have the form $\eul^{2\ii t\theta(k,\xi)}$ or $\eul^{-2\ii t\theta(k,\xi)}$, and the following transformations of the basic RH problem  are guided by the signature structure of $\Im\theta(k,\xi)$.

Since $\theta(k,\xi)=2k^2+4\xi k+\ord(1)$ as $k\to\infty$,  the large $k$ behavior of the signature table for $\Im\theta(k,\xi)$ is the same as for $\Im(4\xi k+2k^2)$. Though the equation $\frac{\dd}{\dd k}\theta(k,\xi)=0$ has two zeros for all $\xi >0$:
\begin{equation}\label{fascrit}
k_1(\xi)=-\frac{1}{2}\left(\xi+\sqrt{\xi^2+2A^2}\right)\quad\text{and}
\quad k_2(\xi)=-\frac{1}{2}\left(\xi-\sqrt{\xi^2+2A^2}\right),
\end{equation}
the signature table of $\Im\theta(k,\xi)$ involves $k_1(\xi)$ only, see Figures \ref{fas_sign_pw} and \ref{fas_sign_plat}. Namely, one can distinguish two cases:
\begin{enumerate}[(1)]
\item 
$\xi\in(A/2,+\infty)$. In this case, the  signature table of $\Im\theta(k,\xi)$ is as in Figure \ref{fas_sign_pw}. The curves separating the domains where $\Im\theta(k,\xi)>0$ and $\Im\theta(k,\xi)<0$ intersect at $k=k_1(\xi)$. 
\item
$\xi\in(0,A/2)$. In this case, the signature table of $\Im\theta(k,\xi)$ is as in Figure \ref{fas_sign_plat}. The curves separating the domains where $\Im\theta(k,\xi)>0$ and $\Im\theta(k,\xi)<0$ intersect at $k=-2\xi$. This is because of
\begin{equation}\label{fascut}
\Im\theta_\pm(k,\xi)=\pm 2(2\xi+k)\sqrt{k^2-A^2},\quad k\in(-A,A).
\end{equation}
\end{enumerate}

\begin{figure}[h]
\begin{minipage}[h]{0.49\linewidth}
\centering{\includegraphics[width=0.99\linewidth]{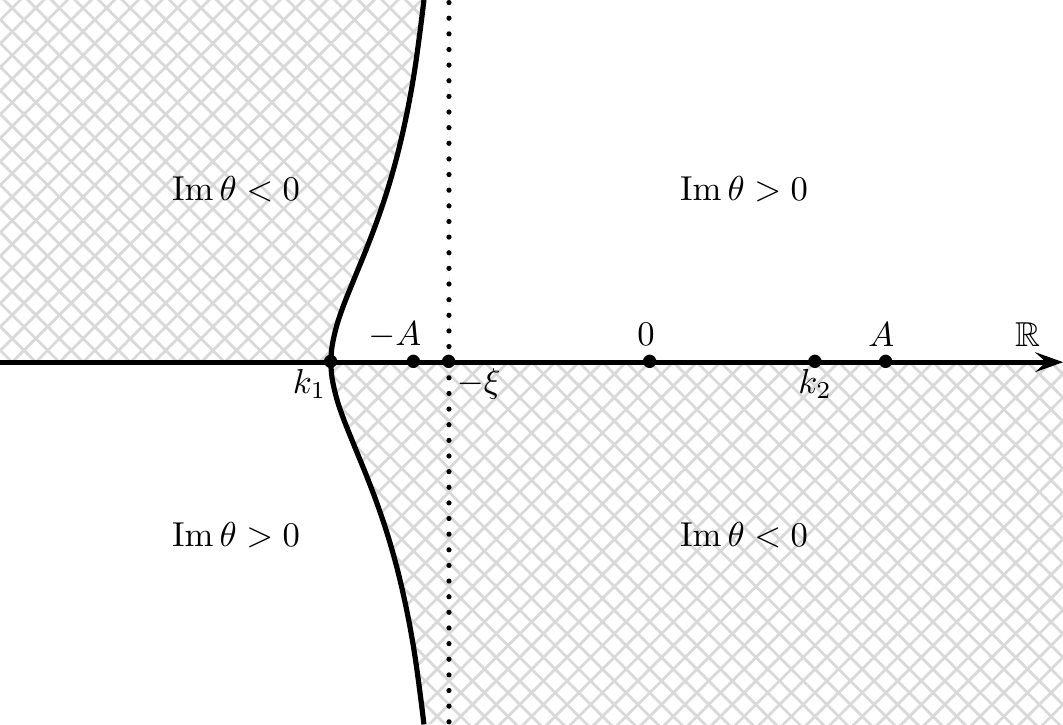}}
\caption{Signature table of $\Im\theta(k,\xi)$ in the modulated wave region $\xi>A/2$.}
\label{fas_sign_pw}
\end{minipage}
\hfill
\begin{minipage}[h]{0.49\linewidth}
\centering{\includegraphics[width=0.99\linewidth]{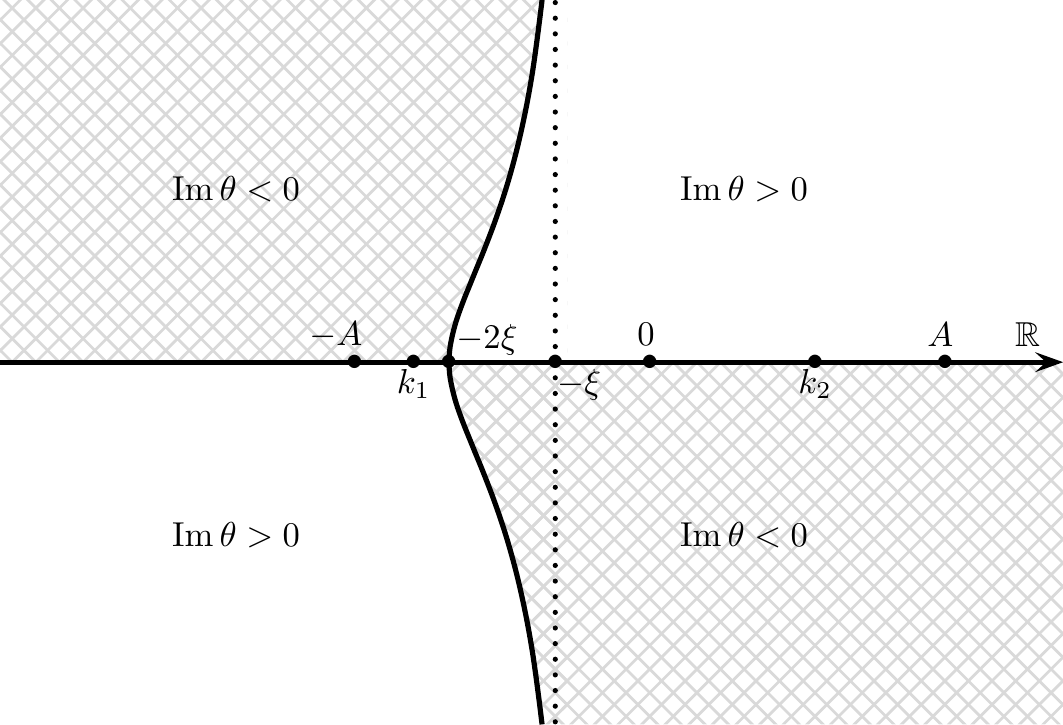}}
\caption{Signature table of $\Im\theta(k,\xi)$ in the central region $0<\xi<A/2$.}
\label{fas_sign_plat}
\end{minipage}
\end{figure}

\subsection{Modulated regions $\abs{\xi}\in(A/2,\infty)$}\label{faspwr}
Taking into account the signature structure of $\Im\theta(k,\xi)$ for $-\xi\in(-\infty,-A/2)$ (see Figure \ref{fas_sign_pw}), we will use two different triangular factorizations of the jump matrix $J(x,t,k)$ for $k\in\D{R}\setminus[-A,A]$ (cf. \cites{DIZ, J15, RS21PD}):
\begin{subequations}\label{fasstr}
\begin{equation}\label{fastr1}
J(x,t,k)=\begin{pmatrix}
	1& 0\\
	\frac{r_1(k)\eul^{2\ii t\theta}}{1+ r_1(k)r_2(k)}& 1\\
	\end{pmatrix}
	\begin{pmatrix}
	1+ r_1(k)r_2(k)& 0\\
	0& \frac{1}{1+ r_1(k)r_2(k)}\\
	\end{pmatrix}
	\begin{pmatrix}
	1& \frac{r_2(k)\eul^{-2\ii t\theta}}{1+ r_1(k)r_2(k)}\\
	0& 1\\
	\end{pmatrix},\, k\in(-\infty,k_1),
\end{equation}
and
\begin{equation}\label{fstr2}
J(x,t,k)=\begin{pmatrix}
	1& r_2(k)\eul^{-2\ii t\theta}\\
	0& 1\\
	\end{pmatrix}
	\begin{pmatrix}
	1& 0\\
	r_1(k)\eul^{2\ii t\theta}& 1\\
	\end{pmatrix},\quad k\in(k_1,-A)\cup(A,\infty).
\end{equation}
\end{subequations}
For getting rid of the diagonal factor in \eqref{fastr1}, we introduce the scalar function $\delta(k,k_1)$ as the solution of the following RH problem:
\begin{equation}
\label{fasdrh}
\begin{aligned}
&\delta_+(k,k_1)=\delta_-(k,k_1)(1+r_1(k)r_2(k)),&& k\in(-\infty,k_1),\\
&\delta(k,k_1)\to 1, && k\to\infty.
\end{aligned}
\end{equation}
The jump function $1+r_1(k)r_2(k)$ in \eqref{fasdrh} is, in general, complex valued for $k\in(-\infty,k_1)$, which is an important difference comparing with the problems for the local equations, where it is real \cites{BKS11, BM17, DZ, J15}. The nonzero imaginary part of $1+r_1(k)r_2(k)$ is responsible for the singularity (or zero, depending on the sign) of $\delta$ at the endpoint $k=k_1$, which follows from the integral representation for $\delta(k,k_1)$ (cf. \cite{RS21PD}):
\begin{equation}\label{fasdelta-int}
\delta(k,k_1)=\exp\left\{\frac{1}{2\pi \ii}\int_{-\infty}^{k_1}
\frac{\ln(1+r_1(\zeta)r_2(\zeta))}{\zeta-k}\,\dd\zeta
\right\}.
\end{equation}
Integrating by parts one concludes that
\begin{equation}\label{fasdelta-singular}
\delta(k,k_1)=(k-k_1)^{\ii\nu(k_1)}\eul^{\chi(k,k_1)},
\end{equation}
where
\begin{align}\label{faschi}
\chi(k,k_1)&\coloneqq-\frac{1}{2\pi\ii}\int_{-\infty}^{k_1}\ln(k-\zeta)\dd\ln(1+ r_1(\zeta)r_2(\zeta)),\\
\label{fasnu}
\nu(k_1)&\coloneqq-\frac{1}{2\pi}\ln(1+r_1(k_1)r_2(k_1))
= -\frac{1}{2\pi}\ln|1+r_1(k_1)r_2(k_1)|-
\frac{\ii}{2\pi}\Delta(k_1),\\
\label{fsDelta-arg}
\Delta(k_1)&\coloneqq\int_{-\infty}^{k_1}\dd\arg(1+ r_1(\zeta)r_2(\zeta)).
\end{align}

To obtain the asymptotics in the modulated regions (see Theorem \ref{fasth1pw} below) we need an additional assumption on the spectral functions (cf.~\cite{RS21PD}):

\begin{assumption}[on the spectral functions $r_1$ and $r_2$]\label{assumpt}
\begin{equation}\label{fasarg-ass}
\int_{-\infty}^{k}\dd\arg(1+r_1(\zeta)r_2(\zeta))\in(-\pi,\pi),\quad \text{for all}\quad k\in(-\infty,-A).
\end{equation}
\end{assumption}

This implies that $\abs{\Im\nu(k_1)}<\frac{1}{2}$ and, consequently, $\delta^{\sigma_3}(k,k_1)$ has a square integrable singularity at $k=k_1$. 

\subsubsection{1st transformation}
Using the function $\delta(k,k_1)$ we make the following transformation of $M(x,t,k)$:
\begin{equation}\label{fasm1def}
M^{(1)}(x,t,k)=M(x,t,k)\delta^{-\sigma_3}(k,k_1),\quad k\in\D{C}\setminus\D{R}.
\end{equation}
Then $M^{(1)}(x,t,k)$ solves the following RH problem:
\begin{subequations}
\begin{alignat}{2}
&M^{(1)}_+(x,t,k)=M^{(1)}_-(x,t,k)J^{(1)}(x,t,k),&\quad& k\in\D{R}\setminus\{\pm A\},\\
&M^{(1)}(x,t,k)=I+\ord(k^{-1}), &&k\to\infty,\\
&M^{(1)}(x,t,k)=
	\ord\left((k\pm A)^{-\frac{1}{4}}\right),
	&& k\to\mp A,\\
	\label{fsm1k1}
&M^{(1)}(x,t,k)=\ord
	\begin{pmatrix}
	(k-k_1)^{p} & (k-k_1)^{-p}\\
	(k-k_1)^{p}&
	(k-k_1)^{-p}
	\end{pmatrix},&& k\to k_1,\,\,p\in(-1/2,1/2),
\end{alignat}
\end{subequations}
where the jump matrix $J^{(1)}(x,t,k)$ has the form
\begin{equation}
J^{(1)}=
\begin{cases}
\begin{pmatrix}
1& 0\\
\frac{r_1(k)\delta_-^{-2}(k,k_1)}{1+ r_1(k)r_2(k)}\eul^{2\ii t\theta}& 1\\
\end{pmatrix}
\begin{pmatrix}
1& \frac{r_2(k)\delta_+^{2}(k,k_1)}{1+ r_1(k)r_2(k)}\eul^{-2\ii t\theta}\\
0& 1\\
\end{pmatrix},& k\in(-\infty,k_1),\\
\begin{pmatrix}
1& r_2(k)\delta^2(k,k_1)\eul^{-2\ii t\theta}\\
0& 1\\
\end{pmatrix}
\begin{pmatrix}
1& 0\\
r_1(k)\delta^{-2}(k,k_1)\eul^{2\ii t\theta}& 1\\
\end{pmatrix}, &k\in(k_1,-A)\cup(A,\infty),\\
\begin{pmatrix}
0 & -\delta^{2}(k,k_1)\\
\delta^{-2}(k,k_1) & 0
\end{pmatrix},& k\in(-A,A).
\end{cases}
\end{equation}
Moreover, $M^{(1)}(x,t,k)$ satisfies singularity conditions
at $k=0$:
\begin{subequations}\label{fask=0(1)RH}
\begin{align}
\lim_{\substack{k\to0,\\k\in\D{C}^+}}k\left(M^{(1)}\right)^{[1]}(x,t,k)&=\frac{\gamma_+}{a_{10}\,\delta^{2}(0,k_1)}\eul^{-2Ax}\left(M^{(1)}\right)^{[2]}_+(x,t,0),\\
\lim_{\substack{k\to0,\\k\in\D{C}^-}}k\left(M^{(1)}\right)^{[2]}(x,t,k)&=\frac{\gamma_-\,\delta^{2}(0,k_1)}{a_{20}}\eul^{-2Ax}\left(M^{(1)}\right)^{[1]}_-(x,t,0).
\end{align}
\end{subequations}

\subsubsection{2nd transformation}  \label{sec:pw2}
Now we are able to get off the real axis and to obtain a RH problem which can be approximated, as $t\to+\infty$, by an exactly solvable problem. We assume that the reflection coefficients $r_j(k)$, $j=1,2$ can be continued into a band containing the real axis (this takes place, for example, when $q_0(x)$ converges exponentially fast to its boundary values). 

Define $M^{(2)}(x,t,k)$ as follows (compare with $M^{(2)}$ in \cite{RS21PD} and $M$ in \cite{J15}):
\[
M^{(2)}=M^{(1)}\times
\begin{cases}
\begin{pmatrix}
1& \frac{-r_2(k)\delta^{2}(k,k_1)}{1+r_1(k)r_2(k)}\eul^{-2\ii t\theta}\\
0& 1\\
\end{pmatrix},\, k\in\hat\Omega_1,&
\begin{pmatrix}
1& 0\\
-r_1(k)\delta^{-2}(k,k_1)\eul^{2\ii t\theta}& 1\\
\end{pmatrix},\, k\in\hat\Omega_2,\\
\begin{pmatrix}
1& r_2(k)\delta^2(k,k_1)\eul^{-2\ii t\theta}\\
0& 1\\
\end{pmatrix},\, k\in\hat\Omega_3,&
\begin{pmatrix}
1& 0\\
\frac{r_1(k)\delta^{-2}(k,k_1)}{1+r_1(k)r_2(k)}\eul^{2\ii t\theta}& 1\\
\end{pmatrix}
,\, k\in\hat\Omega_4,\\
\,I,\quad k\in\hat\Omega_0,
\end{cases}
\]
where $\hat\Omega_j$, $j=0,\dots,4$ are displayed in Figure \ref{fas_cont_pw_1}. Let $\hat\Gamma=\cup_{j=1}^4\hat\gamma_j$ be the contour also shown in Figure \ref{fas_cont_pw_1}. Then $M^{(2)}(x,t,k)$ solves the following RH problem: 
\begin{subequations}
\begin{alignat}{2}
&M^{(2)}_+(x,t,k)=M^{(2)}_-(x,t,k)J^{(2)}(x,t,k),&\quad&k\in\hat\Gamma\cup (-A,A),\\
&M^{(2)}(x,t,k)=I+\ord(k^{-1}), &&k\to\infty,\\
&M^{(2)}(x,t,k)=\ord\left((k\pm A)^{-\frac{1}{4}}\right),&& k\to\mp A,\\
&M^{(2)}(x,t,k)=\ord
\begin{pmatrix}
	(k-k_1)^{p} & (k-k_1)^{-p}\\
	(k-k_1)^{p}&
	(k-k_1)^{-p}
\end{pmatrix},&& k\to k_1,\,\,p\in(-1/2,1/2),
\end{alignat}
\end{subequations}
where, using the relations $r_{1\pm}(k)=r_{2\mp}(k)$ and $\theta_+(k)=-\theta_-(k)$ for $k\in(-A,A)$, one finds that
\begin{equation}
\label{fasJ2}
J^{(2)}=
\begin{cases}
\begin{pmatrix}
0& -\delta^{2}(k,k_1)\\
\delta^{-2}(k,k_1)&0
\end{pmatrix},\,k\in (-A,A);\\
\begin{pmatrix}
1& \frac{r_2(k)\delta^{2}(k,k_1)}{1+r_1(k)r_2(k)}\eul^{-2\ii t\theta}\\
0& 1\\
\end{pmatrix}
,\, k\in\hat\gamma_1;&
\begin{pmatrix}
1& 0\\
r_1(k)\delta^{-2}(k,k_1)\eul^{2\ii t\theta}& 1\\
\end{pmatrix}
,\ k\in\hat\gamma_2;
\\
\begin{pmatrix}
1& -r_2(k)\delta^2(k,k_1)\eul^{-2\ii t\theta}\\
0& 1\\
\end{pmatrix}
,\, k\in\hat\gamma_3;&
\begin{pmatrix}
1& 0\\
\frac{-r_1(k)\delta^{-2}(k,k_1)}{1+r_1(k)r_2(k)}\eul^{2\ii t\theta}& 1\\
\end{pmatrix}
,\, k\in\hat\gamma_4.
\end{cases}
\end{equation}
Using the equalities $r_1(k)=\frac{b_+(0)}{a_{10}k}+\ord(1)$ as $k\to 0$ with $k\in\D{C}^+$, $\theta_+(0,\xi)=\ii A\frac{x}{t}$, and $\gamma_+=b_+(0)$ (see Remark \ref{fasbgamma}), direct calculations show that $M^{(2)}(x,t,k)=\ord(1)$ as $k\to 0$, $k\in\hat\Omega_2$. Similarly, it can be shown that $M^{(2)}(x,t,k)=\ord(1)$ as $k\to 0$, $k\in\hat\Omega_3$. Thus the RH problem for $M^{(2)}$, in contrast to that for $M^{(1)}$, does not involve any singularity conditions at $k=0$.

\begin{figure}[h]
\begin{minipage}[h]{0.49\linewidth}
\centering{\includegraphics[width=0.96\linewidth]{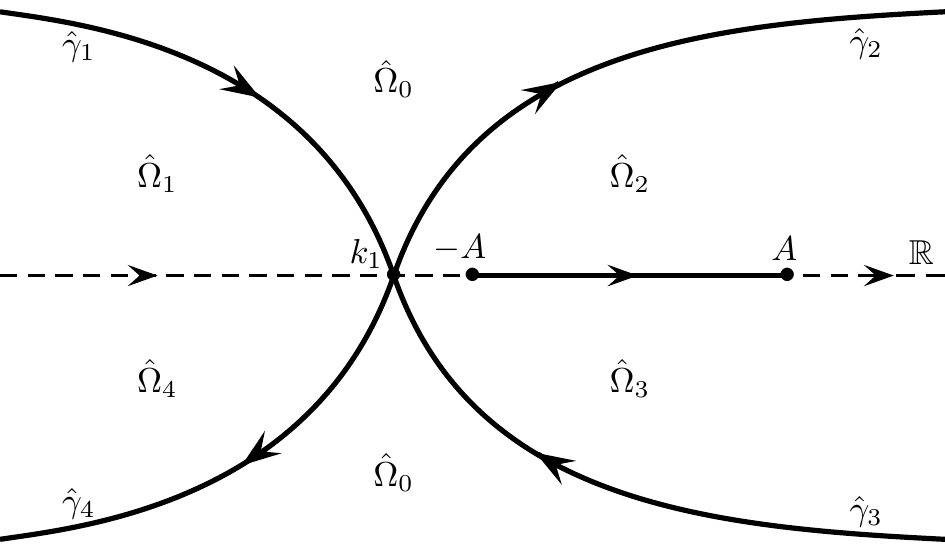}}
\caption{Modulated wave region: contour $\hat\Gamma=\hat\gamma_1\cup\dots\cup\hat\gamma_4$ and domains $\hat\Omega_j$, $j=0,\dots,4$.}\label{fas_cont_pw_1}
\end{minipage}
\hfill
\begin{minipage}[h]{0.49\linewidth}
\centering{\includegraphics[width=0.96\linewidth]{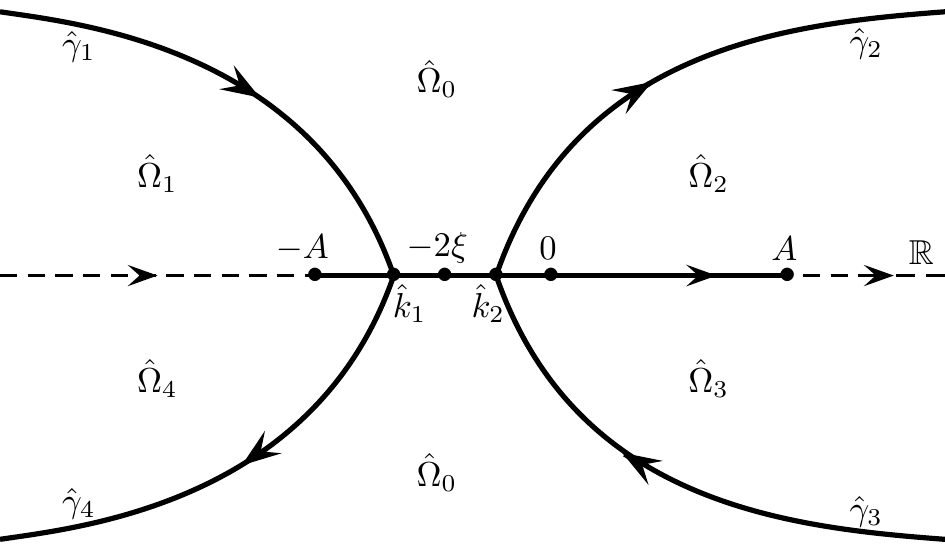}}
\caption{Central region: contour $\hat\Gamma=\hat\gamma_1\cup\dots\cup\hat\gamma_4$ and domains $\hat\Omega_j$, $j=0,\dots,4$.}
\label{fas_cont_cp_1}
\end{minipage}
\end{figure}

In view of the signature table of $\Im\theta(k,\xi)$ (see Figure \ref{fas_sign_pw}), the jump matrix $J^{(2)}(x,t,k)$ decays to the identity matrix for $k\in\hat\Gamma$, uniformly outside any neighborhood of the stationary phase point $k=k_1$. Arguing as, e.g., in \cite{RS21PD}*{Section 3.2}, we eliminate $\delta(k,\xi)$ in the jump for $k\in(-A,A)$ by introducing the scalar function
\begin{equation}\label{fasF}
F(k,k_1)\coloneqq\exp\left\{-\frac{f(k)}{\pi \ii}\int_{-A}^A\frac{\ln\delta(\zeta,k_1)}{f_-(\zeta)(\zeta-k)}\,\dd\zeta\right\},\quad k\in\D{C}\setminus[-A,A].
\end{equation}
This function $F(k,k_1)$ satisfies the jump condition
\begin{equation}
F_+(k,k_1)F_-(k,k_1)=\delta^{2}(k,k_1),\quad k\in(-A,A),
\end{equation}
and is bounded at $k=\pm A$. In order to recover $q(x,t)$ from the solution of the RH problem, we need the large $k$ asymptotics of $F(k,k_1)$:
\begin{equation}
\begin{split}
F(k,k_1)&= \eul^{\ii F_{\infty}(k_1)}+\ord(k^{-1}),\quad
k\to\infty,\\
F_{\infty}(k_1)&\coloneqq-\frac{1}{\pi}\int_{-A}^{A}\frac{\ln\delta(\zeta,k_1)}{f_-(\zeta)}\,\dd\zeta.
\end{split}
\end{equation}
Substituting \eqref{fasdelta-int} into $F_{\infty}(k_1)$, we have that 
\begin{subequations}\label{fasreimFinf}
\begin{align}
&\Re F_{\infty}(k_1)=-\frac{1}{2\pi^2}\int_{-A}^{A}
	\frac{1}{\sqrt{A^2-\zeta^2}}
	\left(\int_{-\infty}^{k_1}
	\frac{\ln|1+r_1(s)r_2(s)|}{s-\zeta}\,\dd s\right)\dd\zeta,\\
&\Im F_{\infty}(k_1)=-\frac{1}{2\pi^2}\int_{-A}^{A}\frac{1}{\sqrt{A^2-\zeta^2}}\left(\int_{-\infty}^{k_1}\frac{\Delta(s)}{s-\zeta}\,\dd s\right)\dd\zeta,
\end{align}
\end{subequations}
where $\Delta(s)$ is given by \eqref{fsDelta-arg} and $\sqrt{A^2-\zeta^2}>0$.

\subsubsection{3rd transformation}  \label{sec:pw3}
Using $F(k,k_1)$, we define $M^{(3)}(x,t,k)$ as follows
\begin{equation}
M^{(3)}(x,t,k)=\eul^{-\ii F_{\infty}(k_1)\sigma_3}
M^{(2)}(x,t,k)F^{\sigma_3}(k,k_1),
\quad 
k\in\D{C}\setminus\left\{\hat\Gamma\cup [-A,A]\right\}.
\end{equation}
Then $M^{(3)}$ satisfies the following RH problem with constant jump across $(-A,A)$:
\begin{subequations}\label{fasM3}
\begin{alignat}{2}
&M^{(3)}_+(x,t,k)=M^{(3)}_-(x,t,k)J^{(3)}(x,t,k),&\quad&
	k\in\hat\Gamma\cup (-A,A),\\
&M^{(3)}(x,t,k)=I+\ord(k^{-1}), &&k\to\infty,\\
&M^{(3)}(x,t,k)=\ord\left((k\pm A)^{-\frac{1}{4}}\right),&& k\to\mp A,\\
&M^{(3)}(x,t,k)=\ord
	\begin{pmatrix}
	(k-k_1)^{p} & (k-k_1)^{-p}\\
	(k-k_1)^{p}&
	(k-k_1)^{-p}
	\end{pmatrix},&&k\to k_1,\ \ p\in(-1/2,1/2),
\end{alignat}
\end{subequations}
with 
\begin{equation}
\label{fasJ3}
J^{(3)}(x,t,k)=
\begin{cases}
-\ii \sigma_2,&k\in (-A,A),\\
F^{-\sigma_3}(k,k_1)J^{(2)}(x,t,k)F^{\sigma_3}(k,k_1),&k\in\hat\Gamma.
\end{cases}
\end{equation}
Since $F(k,k_1)$ is bounded at $k=0$, we have $M^{(3)}(x,t,k)=\ord(1)$ as $k\to 0$. Thus, similarly to the RH problem for $M^{(2)}$, the RH problem for $M^{(3)}$ does not involve any singularity conditions at $k=0$.

The solution $q(x,t)$ of the Cauchy problem \eqref{fasivp} can be expressed in terms of $M^{(3)}(x,t,k)$ as follows:
\begin{subequations}\label{fsasolM3}
\begin{alignat}{2}\label{fssolM3}
q(x,t)&=2\ii \eul^{-2\ii A^2t+2\ii F_\infty(k_1)}\lim_{k\to\infty}kM_{12}^{(3)}(x,t,k),&\quad& x>0,\\
\label{fssol1M3}
q(x,t)&=-2\ii\eul^{-2\ii A^2t+2\ii\overline{F_\infty(k_1)}}\lim_{k\to\infty}\overline{kM_{21}^{(3)}(-x,t,k)},&& x<0.
\end{alignat}
\end{subequations}

\subsubsection{Model RH problem}
Arguing as in \cite{RS21PD}, the RH problem for $M^{(3)}$ can be approximated by a model RH problem whose contour is $(-A,A)$ and whose jump matrix is constant. Using \eqref{fsasolM3}, we are able to obtain an asymptotics of $q(x,t)$ including at least the first decaying term \cite{RS21PD}. For the sake of brevity, we present here, in Theorem \ref{fasth1pw} below, the leading (non-decaying) terms only. 

\begin{theorem}[modulated regions $\abs{\xi}>A/2$]\label{fasth1pw}
Assume that the initial data $q_0(x)$ approaches its boundary values \eqref{fasivp-c} exponentially fast and that the associated spectral functions $a_j(k)$ and $r_j(k)=\frac{b_j(k)}{a_j(k)}$, $j=1,2$ satisfy Assumptions \ref{assumpts} and \ref{assumpt}.

Then the solution $q(x,t)$ of problem \eqref{fasivp} has the following long-time asymptotics along the rays $\xi\equiv\frac{x}{4t}=\const$, uniformly in any compact subset of $\accol{\xi\in\D{R}:\abs{\xi}\in(A/2,+\infty)}$:
\begin{equation}\label{fassolmod}
q(x,t)=\begin{cases}
A\eul^{-2\Im F_{\infty}(k_1(\abs{\xi}))}
\eul^{-2\ii (A^2t-\Re F_{\infty}(k_1(\abs{\xi})))}+E(x,t), &\xi>A/2,\\
-A\eul^{2\Im F_{\infty}(k_1(\abs{\xi}))}
\eul^{-2\ii (A^2t-\Re F_{\infty}(k_1(\abs{\xi})))}+E(x,t), &\xi<-A/2,
\end{cases}
\end{equation}
where $k_1$ and $F_{\infty}(k_1)$ are defined by \eqref{fascrit} and \eqref{fasreimFinf}, respectively, and with error terms $E(x,t)=\ord(t^{-\frac{1}{2}-\Im\nu(k_1(\abs{\xi}))}+t^{-\frac{1}{2}+\Im\nu(k_1(\abs{\xi}))})$.
\end{theorem}

\begin{remark}
In contrast to the plane wave regions for problems for the defocusing NLS equation \cites{B89, EGGK, IU86, J15}, the modulus of the main term in  \eqref{fassolmod} depends on the direction $\xi$. Notice that the absolute value of the main term of the asymptotics in the plane wave regions \cite{RS21PD} and the so-called ``modulated constant'' regions \cites{RSs, RS21CIMP} in problems for the NNLS equation with nonzero symmetric and step-like boundary conditions also depends on the direction $\xi$.
\end{remark}

\subsection{Central region ($\abs{\xi}\in(0,A/2)$)}\label{fas_centr_reg}
For this region, in contrast to the modulated regions (see Section \ref{faspwr}), the sign-changing critical point $k=-2\xi$ lies on the cut $(-A,A)$ (see Figure \ref{fas_sign_plat}). Since $\Im\theta(k,\xi)$ does not vanish on the cut ($\pm\Im\theta_\pm(k,\xi)<0$ for $k\in(-A,-2\xi)$ and $\pm\Im\theta_\pm(k,\xi)>0$ for $k\in(-2\xi,A)$), we are able to obtain the asymptotics with exponential precision (see \cite{IU86} and \cite{J15}*{Section 5.5}). Moreover, no additional conditions on the winding of the argument are needed, because in the central region there is no need to deal with a model problem on the cross.

\subsubsection{1st transformation}

The first transformation is similar to that in the modulated region, but with $\delta(k,-A)$ instead of $\delta(k,k_1)$ (cf. \eqref{fasm1def}):
\begin{equation}\label{fasm1defpl}
M^{(1)}(x,t,k)=M(x,t,k)\delta^{-\sigma_3}(k,-A),\quad k\in\D{C}\setminus\D{R}.
\end{equation}
Then $M^{(1)}(x,t,k)$ solves a RH problem similar to that in the modulated regions, but with, in general, a \emph{strong} singularity at $k=-A$. The form of this singularity depends on whether the quantity $1+r_1(-A)r_2(-A)$ is equal to zero or not (see Remark \ref{faszerr1r2}). Here we only consider the most complicated case, when $1+r_1(-A)r_2(-A)=0$.

Using the results of \cite{G66}*{Sections 8.1 and 8.5} about the behavior of Cauchy-type integrals at the end points and the relation $\ln(-A)=\ln A+\ii\pi$, we have that 
\begin{equation}
\frac{1}{2\pi\ii}\int_{-\infty}^{-A}\frac{\ln\frac{\zeta+A}{\zeta}}{\zeta-k}\,\dd\zeta=\frac{1}{2\pi\ii}\ln A\cdot\ln(k+A)+\frac{1}{4\pi\ii}\ln^2(k+A)+\Phi_{-A}(k),
\end{equation}
where $\Phi_{-A}(k)$ is analytic in a neighborhood of $k=-A$. Since 
\[
\int_{-\infty}^{-A}\dd\arg(1+r_1(\zeta)r_2(\zeta))=
\int_{-\infty}^{-A}\dd\arg\frac{\zeta+A}{\zeta}(1+r_1(\zeta)r_2(\zeta))
\]
and $\ln^2(k+A)=\ln^2|k+A|+\arg^2(k+A)+2\ii\arg(k+A)\cdot\ln(k+A)$, we obtain the following behavior of $\delta(k,-A)$ at $k=-A$:
\begin{equation}\label{fasd-A}
\delta(k,-A)=(k+A)^{\frac{1}{2\pi}(\Delta(-A)+\arg(k+A))}\delta_{-A}(k),
\end{equation}
where $\Delta(-A)$ is given by \eqref{fsDelta-arg} and $\delta_{-A}(k)$ is bounded at $k=-A$. Then $M^{(1)}$ has the following behavior at $k=-A$:
\begin{equation}\label{fasM1-A}
M^{(1)}(x,t,k)=\ord
\begin{pmatrix}
(k+A)^{-\frac{1}{2\pi}(\Delta(-A)+\arg(k+A))-\frac{1}{4}} & 
(k+A)^{\frac{1}{2\pi}(\Delta(-A)+\arg(k+A))-\frac{1}{4}}\\
(k+A)^{-\frac{1}{2\pi}(\Delta(-A)+\arg(k+A))-\frac{1}{4}}& (k+A)^{\frac{1}{2\pi}(\Delta(-A)+\arg(k+A))-\frac{1}{4}}
\end{pmatrix},\quad k\to-A.
\end{equation}

\subsubsection{2nd transformation}

Further, we define $M^{(2)}(x,t,k)$ as in Section \ref{sec:pw2} for the modulated wave case, but with domains $\hat\Omega_j$, $j=0,\dots,4$ displayed in Figure \ref{fas_cont_cp_1}. In that case (see Figure \ref{fas_cont_cp_1}) the points of intersection $\hat k_1$ and $\hat k_2$ of the real axis with $\hat\gamma_1$ and $\hat\gamma_4$, then with $\hat\gamma_2$ and $\hat\gamma_3$ are simply chosen such that $-A<\hat k_1<-2\xi<\hat k_2<0$. Since $1+r_1(k)r_2(k)$ has a simple zero at $k=-A$, choosing $\arg(k+A)\in(2\pi,3\pi)$ for $k\in\D{C}^+$ in the second column of $M^{(1)}$ as $k\to-A$ and $\arg(k+A)\in(-3\pi,-2\pi)$ for $k\in\D{C}^-$ in the first column of $M^{(1)}$ as $k\to-A$ (see \eqref{fasM1-A}) we obtain the behavior \eqref{fasM1-A} for $M^{(2)}$ with $\arg(k+A)\in(-\pi,\pi)$. Moreover, similarly to Section \ref{faspwr}, $k=0$ lies on the boundary of the domains $\hat\Omega_2$ and $\hat\Omega_3$ and thus $M^{(2)}(x,t,k)$ turns to be bounded at $k=0$ as well.

\subsubsection{3rd transformation}

We define $M^{(3)}(x,t,k)$ as in Section \ref{sec:pw3}, but with $F(k,-A)$ instead of $F(k,k_1)$. From \eqref{fasd-A} and \cite{G66}*{Section 8.6} we conclude that $F(k,-A)$ behaves at $k=-A$ as follows:
\begin{equation}\label{fasF-A}
F(k,-A)=(k+A)^{\frac{1}{2\pi}(\Delta(-A)+\arg(k+A))}F_{-A}(k),
\end{equation}
where $F_{-A}(k)$ is bounded at $k=-A$. Therefore, $M^{(3)}(x,t,k)=\ord\bigl((k+A)^{-\frac{1}{4}}\bigr)$ as $k\to-A$. The jump matrix $J^{(3)}$ associated with $M^{(3)}$ is defined similarly to  \eqref{fasJ3}, with $F(k,k_1)$ replaced by $F(k,-A)$ and with the 
contour $\hat\Gamma$ displayed in Figure \ref{fas_cont_cp_1}.

\subsubsection{Model RH problem}
Taking into account that $J^{(3)}(x,t,k)$, $k\in\hat\Gamma$  (see Figure \ref{fas_cont_cp_1}) approaches exponentially fast the identity matrix (as $t\to+\infty$), uniformly with respect to $k\in\hat\Gamma$, we arrive at the following asymptotics for $q(\pm x,t)$:
\begin{subequations}\label{fsasolM4}
\begin{alignat}{2}\label{fssolM4}
q(x,t)&=2\ii\eul^{-2\ii A^2t+2\ii F_\infty(-A)}\lim_{k\to\infty}kM_{12}^\model(k)+\ord(\eul^{-ct}),&\quad& x>0,\,\,t\to+\infty,\\
\label{fssol1M4}
q(-x,t)&=-2\ii\eul^{-2\ii A^2t+2\ii \overline{F_\infty(-A)}}\lim_{k\to\infty}\overline{kM_{21}^\model(k)}+\ord(\eul^{-ct}),&& x>0,\,\,t\to+\infty,
\end{alignat}
\end{subequations}
with some $c>0$, and where $M^\model(k)$ is analytic in $\D{C}\setminus[-A,A]$ and solves the following RH problem with constant jump matrix across the contour $(-A,A)$:
\begin{subequations}\label{fasM4}
\begin{alignat}{2}
&M_+^\model(k)=-\ii M_-^\model(k)\sigma_2,&\quad&k\in (-A,A),\\
&M^\model(k)=I+\ord(k^{-1}), &&k\to\infty,\\
&M^\model(k)=\ord\left((k\pm A)^{-\frac{1}{4}}\right),&& k\to\mp A.
\end{alignat}
\end{subequations}
From \eqref{fassymE12} it follows that $M^\model(k)=\C{E}_2(k)$. Combining this with \eqref{fsasolM4}, we arrive at

\begin{theorem}[unmodulated regions $0<\abs{\xi}<A/2$]\label{fasth2cpr}
Assume that the initial data $q_0(x)$ approaches exponentially fast its boundary values \eqref{fasivp-c} and that the associated spectral functions $a_j(k)$ and $r_j(k)=\frac{b_j(k)}{a_j(k)}$, $j=1,2$ satisfy Assumptions \ref{assumpts}. 

Then the solution $q(x,t)$ of problem \eqref{fasivp} has the following long-time asymptotics along the rays $\xi=\frac{x}{4t}=\const$, uniformly in any compact subset of $\accol{\xi\in\D{R}:\abs{\xi}\in(0,A/2)}$:
\begin{equation}\label{fassolcpr12}
q(x,t)=\begin{cases}
A\eul^{-2\Im F_{\infty}(-A)}
\eul^{-2\ii (A^2t-\Re F_{\infty}(-A))}+\ord(\eul^{-ct}), &0<\xi<A/2,\\
-A\eul^{2\Im F_{\infty}(-A)}
\eul^{-2\ii (A^2t-\Re F_{\infty}(-A))}+\ord(\eul^{-ct}), &-A/2<\xi<0,
\end{cases}
\end{equation}
with some $c>0$ independent of $\xi$. Here $F_{\infty}(-A)$ is given by \eqref{fasreimFinf} with $k_1=-A$.
\end{theorem}

\begin{remark}
The asymptotics in the central (unmodulated) regions is established without additional restrictions on the winding of the argument of the spectral data (cf. Theorem \ref{fasth1pw} and, e.g., \cites{RS21PD, RS21CIMP}). To the best of our knowledge, it is the first discovered zone for nonlocal integrable equations where the asymptotics of the solution does not depend on the behavior of the argument of a dedicated spectral function.
\end{remark}

\begin{remark}
The asymptotics of $q(x,t)$ for $\xi\in(-A/2,0)$ and $\xi\in(0,A/2)$ does not depend on the direction $\xi$. However, both $|q(x,t)|$ and $\arg q(x,t)$ depend on the initial data through $F_\infty(-A)$. 
	
The central region can be compared with the central plateau zone for the defocusing NLS equation, where the asymptotics is also obtained with exponential precision, but the modulus of the solution does not depend on the initial data \cites{B89, EGGK, IU86, J15}.
\end{remark}

\begin{remark}
Since $k_1(\frac{A}{2})=-A$, the main terms in the unmodulated regions, see \eqref{fassolcpr12}, match those in the modulated regions (see \eqref{fassolmod}) at $\xi=\pm\frac{A}{2}$.
\end{remark}

\begin{remark}
The asymptotic formulas \eqref{fassolcpr12} do not match as $\xi\to \pm0$. However, in the central  region, the solution $q(x,t)$ can approach a tanh-like function as $t\to+\infty$ (see Theorem \ref{fasthtrans} below).
\end{remark}

\subsection{Transition at $\xi=0$}\label{fastrans}
In this section we analyse the asymptotics of the solution as $\xi\to\pm0$. For this, we consider $(x,t)$ with $x=x_0>0$ \emph{fixed} and $t\to+\infty$.

\subsubsection{First transformations}
We perform three transformations of the basic RH problem similar to those made in Section \ref{fas_centr_reg}. However, since $\xi\to+0$, we choose the contour $\hat\Gamma$ (see Figure \ref{fas_cont_tr_1}) such that its points of intersection $\hat{k}_1$ and $\hat{k}_2$ with the real axis satisfy $-A<\hat{k}_1<0<\hat{k}_2<A$.
\begin{figure}[h]
\begin{minipage}[h]{0.99\linewidth}
\centering{\includegraphics[width=0.5\linewidth]{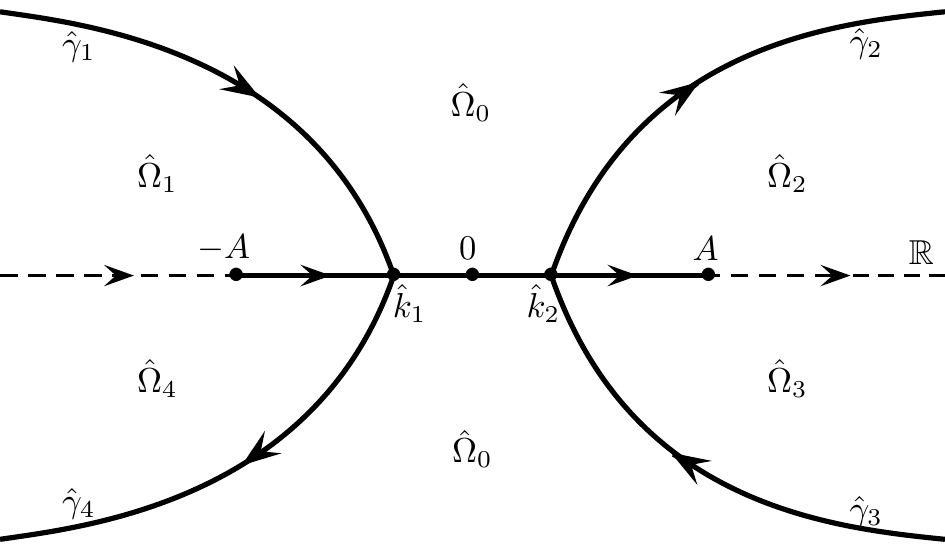}}
\caption{Transition region: contour $\hat\Gamma=\hat\gamma_1\cup\dots\cup\hat\gamma_4$ and domains $\hat\Omega_0,\dots,\hat\Omega_4$.}
\label{fas_cont_tr_1}
\end{minipage}
\end{figure}

In contrast to the cases presented in Sections
\ref{faspwr} and \ref{fas_centr_reg}, now the point $k=0$ lies on the boundary of $\hat\Omega_0$.
It follows that the RH problems for both $M^{(2)}(x,t,k)$ and 
$M^{(3)}(x,t,k)$ involve singularity conditions at $k=0$;
particularly, these conditions for  $M^{(3)}$ read as follows:
\begin{subequations}\label{fask=0M(3)RH}
\begin{align}
&\lim_{\substack{k\to 0,\\k\in\D{C}^+}}
	k\bigl(M^{(3)}\bigr)^{[1]}(x,t,k)
	=\frac{\gamma_+\,F_+^2(0,-A)}{a_{10}\,\delta^{2}(0,-A)}
	\eul^{-2Ax}
	\bigl(M^{(3)}\bigr)^{[2]}_+(x,t,0),\\
&\lim_{\substack{k\to 0,\\k\in\D{C}^-}}
	k\bigl(M^{(3)}\bigr)^{[2]}(x,t,k)
	=\frac{\gamma_-\,\delta^{2}(0,-A)}{a_{20}\,F_-^{2}(0,-A)}
	\eul^{-2Ax}
	\bigl(M^{(3)}\bigr)^{[1]}_-(x,t,0).
\end{align}
\end{subequations}

\subsubsection{Model RH problem}
The solution $M^{(3)}(x,t,k)$ of the RH problem relative to the contour $\hat\Gamma\cup(-A,A)$ (see Figure \ref{fas_cont_tr_1}) can be approximated by the solution $M^\model(x,k)$ of a model problem, which is as follows (cf.~\eqref{fasMsol} and \eqref{fask=0sol}):
\begin{subequations}\label{fasM_tilde}
\begin{alignat}{2}
M_+^\model(x,k)&=-\ii M_-^\model(x,k)\sigma_2,&\quad&k\in (-A,A),\\
M^\model(x,k)&=I+\ord(k^{-1}), &&k\to\infty,\\
M^\model(x,k)&=\ord\bigl((k\pm A)^{-\frac{1}{4}}\bigr),&& k\to\mp A,
\end{alignat}
\end{subequations}
with singularity conditions at $k=0$:
\begin{subequations}\label{fask=0(tilde)RH}
\begin{align}
&\lim_{\substack{k\to 0,\\k\in\D{C}^+}}
kM^{\model[1]}(x,k)=\frac{\gamma_+\,F_+^2(0,-A)}{a_{10}\,\delta^{2}(0,-A)}\eul^{-2Ax}M^{\model[2]}_+(x,0),\\
&\lim_{\substack{k\to 0,\\k\in\D{C}^-}}kM^{\model[2]}(x,k)=\frac{\gamma_-\,\delta^{2}(0,-A)}{a_{20}\,F_-^{2}(0,-A)}\eul^{-2Ax}M_-^{\model[1]}(x,0).
\end{align}
\end{subequations}
Indeed, writing
\begin{equation}\label{fasM^3trans}
M^{(3)}(x,t,k)=M^\err(x,t,k)M^\model(x,k),
\end{equation}
$M^\err$ satisfies the following RH problem on the contour $\hat\Gamma$:
\begin{subequations}\label{fasM_err}
\begin{alignat}{2}
M^\err_+(x,t,k)&= M^\err_-(x,t,k)J^\err(x,t,k),&\quad&k\in\hat\Gamma,\\
M^\err(x,k)&=I+\ord(k^{-1}), &&k\to\infty,\\
M^\err(x,k)&=\ord\left((k\pm A)^{-\frac{1}{2}}\right),&& k\to\mp A,
\end{alignat}
\end{subequations}
where $J^\err(x,t,k)$, $k\in\hat\Gamma$ can be uniformly estimated with exponentially small error for large $t$:
\begin{equation}\label{fasJerr}
J^\err(x,t,k)=M^\model(x,k)(I+\ord(\eul^{-ct}))(M^\model)^{-1}(x,k),\quad t\to+\infty,
\end{equation}
with some $c>0$ which does not depend on $x$. It follows that for all $x$ such that $2A+\ii d(A)\eul^{-2Ax}\neq 0$ (see \eqref{fasMssol-1}),
\begin{equation}\label{fasMerr}
M_1^\err(x,t)\coloneqq\lim_{k\to\infty}k\bigl(M^\err(x,t,k)-I\bigr)
=\frac{\ord(\eul^{-ct})}{2A+\ii d(A)\eul^{-2Ax}},\quad t\to+\infty,
\end{equation}
where $\ord(\eul^{-ct})$ is independent of $x$ and
\begin{equation}\label{fasdA}
d(A)\coloneqq\frac{\gamma_+\,F_+^2(0,-A)}{a_{10}\,\delta^{2}(0,-A)},
\end{equation}
with $\delta(k,-A)$ and $F(k,-A)$ given by \eqref{fasdelta-int} and \eqref{fasF}, respectively. From \eqref{fasM_tilde} and \eqref{fasMerr} we conclude that $q(x,t)$ and $q(-x,t)$ can be found in terms of the solution $M^\model(x,k)$ as follows:
\begin{subequations}\label{fsasolM_tilde}
\begin{alignat}{2}\label{fssolM_tilde}
q(x,t)&=2\ii \eul^{-2\ii A^2t+2\ii F_\infty(-A)}\lim_{k\to\infty}k\tilde{M}_{12}(x,k)+\ord(\eul^{-ct}),&\quad& x>0,\,\,t\to+\infty,\\
\label{fssol1M_tilde}
q(-x,t)&=-2\ii \eul^{-2\ii A^2t+2\ii\overline{F_\infty(-A)}}\lim_{k\to\infty}\overline{k\tilde{M}_{21}(x,k)}+\ord(\eul^{-ct}),&& x>0,\,\,t\to+\infty.
\end{alignat}
\end{subequations}
Then, arguing as in Section \ref{fasonesol}, we can explicitly solve the RH problem for $M^\model(x,k)$ and thus arrive at

\begin{theorem}[transition at $\xi=0$]\label{fasthtrans}
Assume that the initial data $q_0(x)$ approaches exponentially fast its boundary values \eqref{fasivp-c} and that the associated spectral functions $a_j(k)$ and $r_j(k)=\frac{b_j(k)}{a_j(k)}$, $j=1,2$ satisfy Assumptions \ref{assumpts}. 

Then the solution $q(x,t)$ of problem \eqref{fasivp} has the following asymptotics  as $t\to+\infty$ along the rays $x=\const$, excluding $x=0$ and also $x=x^\prime\coloneqq\frac{1}{2A}\ln\frac{-\ii d(A)}{2A}$ if $x^\prime$ is real and positive, and $x=x^{\prime\prime}\coloneqq-\frac{1}{2A}\ln\frac{\ii\overline{d(A)}}{2A}$ if $x^{\prime\prime}$ is real and negative:
\begin{equation}\label{fassoltrans}
q(x,t)=\begin{cases}
A\eul^{-2\Im F_{\infty}(-A)}\eul^{-2\ii (A^2t-\Re F_{\infty}(-A))}\cdot\frac{2A-\ii d(A)\eul^{-2Ax}}{2A+\ii d(A)\eul^{-2Ax}}+\ord(\eul^{-ct}), & x>0,\\
-A\eul^{2\Im F_{\infty}(-A)}\eul^{-2\ii (A^2t-\Re F_{\infty}(-A))}\cdot\frac{2A\eul^{-2Ax}+\ii \overline{d(A)}}
{2A\eul^{-2Ax}-\ii\overline{d(A)}}+\ord(\eul^{-ct}), & x<0,
\end{cases}
\end{equation}
with some $c>0$ independent of $x$. Here $F_{\infty}(-A)$ and $d(A)$ are given by \eqref{fasreimFinf} and \eqref{fasdA}, respectively.
\end{theorem}

\begin{remark}
As $x\to\pm\infty$, the main terms in \eqref{fassoltrans} match those  in \eqref{fassolcpr12}.
\end{remark}

\begin{remark}
The main term of the asymptotics in \eqref{fassoltrans} is continuous at $x=0$ only if $d(A)$ and $\Im F_{\infty}(-A)$ satisfies one of the two conditions:
\begin{itemize}
\item 
$\Im F_{\infty}(-A)=0$ and $|d(A)|=2A$ with $d(A)\neq 2\ii A$,
\item
$d(A)=-2\ii A$ (without condition on $\Im F_{\infty}(-A)$).
\end{itemize}
\end{remark}
	
\appendix
\section{Proof of Proposition \ref{fasbgiv}} \label{appendix}
\begin{proof}[Proof of item (ii)]
Substituting $q_{0,R}(x)$ with $R=0$ (see \ref{fasstepiv}) to \eqref{fasPsi}, we obtain that $\Psi_j(0,0,k)=\C{E}_j(k)$, $j=1,2$.
Using \eqref{fassrpsi}, we have $S(k)=\C{E}_2^{-1}(k)\C{E}_1(k)$, which implies \eqref{fasR=0} in view of \eqref{fsscatt}.
\renewcommand{\qed}{}
\end{proof}
\begin{proof}[Proof of item (i)]
For the initial data $q_{0,R}(x)$ with $R>0$, from the integral representations \eqref{fasPsi} we have that
\begin{equation}\label{faspsi_2R}
\Psi_2(R,0,k)=\C{E}_2(k)
\end{equation}
and that the $(11)$ and 
$(12)$ entries of $\Psi_1(x,0,k)$ for  $x\in[-R,R]$ satisfy the following integral equations:
\begin{subequations}\label{faspsi11}
\begin{align}
\label{faspsi11a}
&(\Psi_{1})_{11}(x,0,k)=
e_1(k)+2Ae_1(k)e_2(k)\int_{-R}^x\left(1-\eul^{2\ii f(k)(x-y)}\right)
(\Psi_{1})_{11}(y,0,k)\,\dd y, &&x\in[-R,R],\\
\label{faspsi11b}
&(\Psi_{1})_{12}(x,0,k)=
-e_2(k)-2Ae_1(k)e_2(k)\int_{-R}^x\left(1-\eul^{-2\ii f(k)(x-y)}\right)
(\Psi_{1})_{12}(y,0,k)\,\dd y, &&x\in[-R,R],
\end{align}
\end{subequations}
where 
\begin{equation}
e_1(k)\coloneqq\frac{1}{2}\left(w(k)+\frac{1}{w(k)}\right),\quad
e_2(k)\coloneqq\frac{\ii}{2}\left(w(k)-\frac{1}{w(k)}\right),
\end{equation}
with $w(k)$  given in \eqref{fasK}.
The entries $(\Psi_1)_{21}(x,0,k)$ and $(\Psi_1)_{22}(x,0,k)$ can be expressed in terms of $(\Psi_1)_{11}(x,0,k)$ and $(\Psi_1)_{12}(x,0,k)$ as follows:
\begin{subequations}\label{faspsi21}
\begin{align}
\label{faspsi21a}
&(\Psi_{1})_{21}(x,0,k)=
e_2(k)
+2A\int_{-R}^x\left(e_2^2(k)+e_1^2(k)\eul^{2\ii f(k)(x-y)}\right)
(\Psi_{1})_{11}(y,0,k)\,\dd y,&&x\in[-R,R],\\
\label{faspsi22b}
&(\Psi_{1})_{22}(x,0,k)=
e_1(k)
+2A\int_{-R}^x\left(e_1^2(k)
+e_2^2(k)\eul^{-2\ii f(k)(x-y)}\right)
(\Psi_{1})_{12}(y,0,k)\,\dd y,&&x\in[-R,R].
\end{align}
\end{subequations}
In order to find $\Psi_1(R,0,k)$, we first solve the integral equations \eqref{faspsi11} and then substitute the solutions into \eqref{faspsi21} with $x=R$.
Using the equality $e_1(k)e_2(k)=-\frac{\ii A}{2f(k)}$, equation \eqref{faspsi11a} can be reduced to the following Cauchy problem for a linear ordinary differential equation:
\begin{equation}
\label{faspsi11adiff}
\begin{cases}
\frac{\dd^2}{\dd x^2}(\Psi_1)_{11}(x,0,k)-2\ii f(k)\frac{\dd}{\dd x}(\Psi_1)_{11}(x,0,k)+2A^2(\Psi_1)_{11}(x,0,k)=0,&x\in[-R,R],\\
(\Psi_1)_{11}(-R,0,k)=e_1(k),\quad\frac{\dd}{\dd x}(\Psi_1)_{11}(-R,0,k)=0.&
\end{cases}
\end{equation}
The solution of \eqref{faspsi11adiff} has the form
\begin{equation}\label{faspsi11asol}
(\Psi_1)_{11}(x,0,k)=
\frac{\ii e_1(k)\lambda_2(k)}{2h(k)}
\eul^{\lambda_1(k)(x+R)}-
\frac{\ii e_1(k)\lambda_1(k)}{2h(k)}
\eul^{\lambda_2(k)(x+R)},\quad x\in[-R,R],
\end{equation}
where $h(k)$, $\lambda_j(k)$, $j=1,2$ are given by \eqref{fash} and \eqref{faslambda} respectively. Then, substituting \eqref{faspsi11asol} into \eqref{faspsi21a} and using the relations $\frac{\lambda_1(k)}{\lambda_2(k)}=-\frac{f(k)h(k)+k^2}{A^2}$, $\frac{\lambda_2(k)}{\lambda_1(k)}=\frac{f(k)h(k)-k^2}{A^2}$ and $\frac{e_1^2(k)}{e_2^2(k)}=-\frac{(k+f(k))^2}{A^2}$, we obtain:
\begin{align}
(\Psi_1)_{21}(R,0,k)&=e_2(k)
+\ii A\frac{e_1(k)e_2^2(k)}{h(k)}
\left(
\frac{\lambda_2(k)}{\lambda_1(k)}
\left(\eul^{2\lambda_1(k)R}-1\right)
-\frac{\lambda_1(k)}{\lambda_2(k)}
\left(\eul^{2\lambda_2(k)R}-1\right)
\right)\notag\\
&\qquad+\ii A\frac{e_1^3(k)}{h(k)}
\left(
\eul^{2\lambda_2(k)R}-\eul^{2\lambda_1(k)R}\right)\notag\\
&=\frac{A^2e_2(k)}{2f(k)h(k)}
\left(
\eul^{2\lambda_1(k)R}
\left(\frac{\lambda_2(k)}{\lambda_1(k)}
-\frac{e_1^2(k)}{e_2^2(k)}\right)
-\eul^{2\lambda_2(k)R}
\left(\frac{\lambda_1(k)}{\lambda_2(k)}
-\frac{e_1^2(k)}{e_2^2(k)}\right)
\right)\notag\\
&=\frac{e_2(k)}{2h(k)}
\left(
\eul^{2\lambda_1(k)R}(2k-\ii\lambda_1(k))
-\eul^{2\lambda_2(k)R}(2k-\ii\lambda_2(k))
\right).
\end{align}

Similarly, from the integral equation \eqref{faspsi11b} we deduce that
\begin{equation}\label{faspsi12asol}
(\Psi_1)_{12}(x,0,k)=
\frac{\ii e_2(k)\lambda_1(k)}{2h(k)}
\eul^{-\lambda_2(k)(x+R)}-
\frac{\ii e_2(k)\lambda_2(k)}{2h(k)}
\eul^{-\lambda_1(k)(x+R)},\quad x\in[-R,R],
\end{equation}
and, consequently, from \eqref{faspsi22b} we have (here we use
$\frac{e_2^2(k)}{e_1^2(k)}=-\frac{(f(k)-k)^2}{A^2}$)
\begin{align}
(\Psi_1)_{22}(R,0,k)&=e_1(k)
+\ii A\frac{e_1^2(k)e_2(k)}{h(k)}
\left(
\frac{\lambda_2(k)}{\lambda_1(k)}
\left(\eul^{-2\lambda_1(k)R}-1\right)
-\frac{\lambda_1(k)}{\lambda_2(k)}
\left(\eul^{-2\lambda_2(k)R}-1\right)
\right)\notag\\
&\qquad+\ii A\frac{e_2^3(k)}{h(k)}
\left(
\eul^{-2\lambda_2(k)R}-\eul^{-2\lambda_1(k)R}
\right)\notag\\
&=\frac{A^2e_1(k)}{2f(k)h(k)}
\left(
\eul^{-2\lambda_1(k)R}
\left(\frac{\lambda_2(k)}{\lambda_1(k)}
-\frac{e_2^2(k)}{e_1^2(k)}\right)
-\eul^{-2\lambda_2(k)R}
\left(\frac{\lambda_1(k)}{\lambda_2(k)}
-\frac{e_2^2(k)}{e_1^2(k)}\right)
\right)\notag\\
\label{faspsi22asol}
&=\frac{e_1(k)}{2h(k)}
\left(
\eul^{-2\lambda_2(k)R}(2k+\ii\lambda_2(k))
-\eul^{-2\lambda_1(k)R}(2k+\ii\lambda_1(k))
\right).
\end{align}
Finally, substituting \eqref{faspsi_2R} and \eqref{faspsi11asol}--\eqref{faspsi22asol} into 
\begin{equation}\label{fasSstep}
S(k)=\eul^{\ii Rf(k)\sigma_3}\Psi_2^{-1}(R,0,k)\Psi_1(R,0,k)\eul^{-\ii Rf(k)\sigma_3}
\end{equation}
and using equalities $e_1^2(k)=\frac{f(k)+k}{2f(k)}$ and $e_2^2(k)=\frac{f(k)-k}{2f(k)}$, we arrive at \eqref{fasR>0}.
\renewcommand{\qed}{}
\end{proof}
\begin{proof}[Proof of item (iii)]
Let the entries of the $2\times2$ matrix $\hat{\Psi}_1(x,k)$ satisfy \eqref{faspsi11} and \eqref{faspsi21} for $x\in[R,-R]$ (recall that here $R<0$). Then from the integral representation for $\Psi_2(x,0,k)$, see \eqref{fasPsi}, we conclude that the entries of $\Psi_2(x,0,k)$ can be found via $\hat{\Psi}_1(x,k)$ as follows:
\begin{equation}\label{fasPsi_2ijR}
\begin{split}
&(\Psi_2)_{11}(x,0,k)=(\hat{\Psi}_1)_{11}(x,k),\quad\,\,\,\,
	(\Psi_2)_{12}(x,0,k)=-(\hat{\Psi}_1)_{12}(x,k),\\
&(\Psi_2)_{21}(x,0,k)=-(\hat{\Psi}_1)_{21}(x,k),\quad
	(\Psi_2)_{22}(x,0,k)=(\hat{\Psi}_1)_{22}(x,k).
\end{split}
\end{equation}
Therefore, using the expressions for the entries of the matrix $\hat\Psi_1(R,k)$ obtained in the proof of item (i), we obtain $\Psi_2(R,0,k)$. Since $\Psi_1(R,0,k)=\C{E}_1(k)$, from \eqref{fasSstep} and \eqref{fasPsi_2ijR} we have \eqref{fasR<0}.
\renewcommand{\qed}{}
\end{proof}


\begin{bibdiv}
\begin{biblist}
\bib{ALM18}{article}{
   author={Ablowitz, Mark J.},
   author={Luo, Xu-Dan},
   author={Musslimani, Ziad H.},
   title={Inverse scattering transform for the nonlocal nonlinear
   Schr\"{o}dinger equation with nonzero boundary conditions},
   journal={J. Math. Phys.},
   volume={59},
   date={2018},
   number={1},
   pages={011501, 42},
}
\bib{AMP}{article}{
   author={Ablowitz, Mark J.},
   author={Musslimani, Ziad H.},
   title={Integrable nonlocal nonlinear Schr\"odinger equation},
   journal={Phys. Rev. Lett.},
   volume={110},
   date={2013},
   pages={064105},
}
\bib{BB}{article}{
   author={Bender, Carl M.},
   author={Boettcher, Stefan},
   title={Real spectra in non-Hermitian Hamiltonians having $\scr{PT}$
   symmetry},
   journal={Phys. Rev. Lett.},
   volume={80},
   date={1998},
   number={24},
   pages={5243--5246},
}
\bib{BHH}{article}{
   author={Bender, Carl M.},
   author={Holm, Darryl D.},
   author={Hook, Daniel W.},
   title={Complexified dynamical systems},
   journal={J. Phys. A},
   volume={40},
   date={2007},
   number={3},
   pages={F793--F804},
}	
\bib{B89}{article}{
   author={Bikbaev, R. F.},
   title={Diffraction in a defocusing nonlinear medium},
   language={Russian, with English summary},
   journal={Zap. Nauchn. Sem. Leningrad. Otdel. Mat. Inst. Steklov.
   (LOMI)},
   volume={179},
   date={1989},
   number={Mat. Vopr. Teor. Rasprostr. Voln. 19},
   pages={13, 23--31, 187},
   translation={
      journal={J. Soviet Math.},
      volume={57},
      date={1991},
      number={3},
      pages={3078--3083},
   },
}	
\bib{BK14}{article}{
   author={Biondini, Gino},
   author={Kova\v{c}i\v{c}, Gregor},
   title={Inverse scattering transform for the focusing nonlinear
   Schr\"{o}dinger equation with nonzero boundary conditions},
   journal={J. Math. Phys.},
   volume={55},
   date={2014},
   number={3},
   pages={031506, 22},
}
\bib{BM17}{article}{
   author={Biondini, Gino},
   author={Mantzavinos, Dionyssios},
   title={Long-time asymptotics for the focusing nonlinear Schr\"odinger
   equation with nonzero boundary conditions at infinity and asymptotic
   stage of modulational instability},
   journal={Comm. Pure Appl. Math.},
   volume={70},
   date={2017},
   number={12},
   pages={2300--2365},
}
\bib{BKS11}{article}{
   author={Boutet de Monvel, Anne},
   author={Kotlyarov, Vladimir P.},
   author={Shepelsky, Dmitry},
   title={Focusing NLS equation: long-time dynamics of step-like initial
   data},
   journal={Int. Math. Res. Not. IMRN},
   date={2011},
   number={7},
   pages={1613--1653},
}
\bib{CJ16}{article}{
   author={Cuccagna, Scipio},
   author={Jenkins, Robert},
   title={On the asymptotic stability of $N$-soliton solutions of the
   defocusing nonlinear Schr\"{o}dinger equation},
   journal={Commun. Math. Phys.},
   volume={343},
   date={2016},
   number={3},
   pages={921--969},
}
\bib{DIZ}{article}{
   author={Deift, P. A.},
   author={Its, A. R.},
   author={Zhou, X.},
   title={Long-time asymptotics for integrable nonlinear wave equations},
   conference={
      title={Important developments in soliton theory},
   },
   book={
      series={Springer Ser. Nonlinear Dynam.},
      publisher={Springer, Berlin},
   },
   date={1993},
   pages={181--204},
}
\bib{DVZ94}{article}{
   author={Deift, P.},
   author={Venakides, S.},
   author={Zhou, X.},
   title={The collisionless shock region for the long-time behavior of
   solutions of the KdV equation},
   journal={Comm. Pure Appl. Math.},
   volume={47},
   date={1994},
   number={2},
   pages={199--206},
}
\bib{DVZ97}{article}{
   author={Deift, P.},
   author={Venakides, S.},
   author={Zhou, X.},
   title={New results in small dispersion KdV by an extension of the
   steepest descent method for Riemann--Hilbert problems},
   journal={Internat. Math. Res. Notices},
   date={1997},
   number={6},
   pages={286--299},
}
\bib{DZ}{article}{
   author={Deift, P.},
   author={Zhou, X.},
   title={A steepest descent method for oscillatory Riemann--Hilbert
   problems. Asymptotics for the MKdV equation},
   journal={Ann. of Math. (2)},
   volume={137},
   date={1993},
   number={2},
   pages={295--368},
}
\bib{DPMV13}{article}{
   author={Demontis, F.},
   author={Prinari, B.},
   author={van der Mee, C.},
   author={Vitale, F.},
   title={The inverse scattering transform for the defocusing nonlinear
   Schr\"{o}dinger equations with nonzero boundary conditions},
   journal={Stud. Appl. Math.},
   volume={131},
   date={2013},
   number={1},
   pages={1--40},
}
\bib{EGGK}{article}{
   author={\`El\cprime, G. A.},
   author={Geogjaev, V. V.},
   author={Gurevich, A. V.},
   author={Krylov, A. L.},
   title={Decay of an initial discontinuity in the defocusing NLS
   hydrodynamics},
   note={The nonlinear Schr\"{o}dinger equation (Chernogolovka, 1994)},
   journal={Phys. D},
   volume={87},
   date={1995},
   number={1-4},
   pages={186--192},
}	
\bib{EMK18}{article}{
   author={El-Ganainy, R.},
   author={Makris, K. G.},
   author={Khajavikhan, M.},
   author={Musslimani, Ziad H.},
   author={Rotter, S.},
   author={Christodoulides, D. N.},
   title={Non-Hermitian physics and PT symmetry},
   journal={Nature Physics},
   volume={14},
   date={2018},
   pages={11--19},
}	
\bib{AMFL}{article}{
   author={Feng, Bao-Feng},
   author={Luo, Xu-Dan},
   author={Ablowitz, Mark J.},
   author={Musslimani, Ziad H.},
   title={General soliton solution to a nonlocal nonlinear Schr\"{o}dinger
   equation with zero and nonzero boundary conditions},
   journal={Nonlinearity},
   volume={31},
   date={2018},
   number={12},
   pages={5385--5409},
}
\bib{FLQ}{article}{
   author={Fromm, Samuel},   
   author={Lenells, Jonatan},
   author={Quirchmayr, Ronald},
   title={The defocusing nonlinear Schr\"odinger equation with step-like 
   oscillatory initial data},
   date={2011},
   eprint={https://arxiv.org/abs/2104.03714},
}			
\bib{GA}{article}{
   author={Gadzhimuradov, T. A.},
   author={Agalarov, A. M.},
   title={Towards a gauge-equivalent magnetic structure of the nonlocal 
   nonlinear Schr\"odinger equation},
   journal={Phys. Rev. A},
   volume={93},
   date={2016},
   number={6},
   pages={062124},
}			
\bib{G66}{book}{
   author={Gakhov, F. D.},
   title={Boundary value problems},
   note={Translated from the Russian;
   Reprint of the 1966 translation},
   publisher={Dover Publications, Inc., New York},
   date={1990},
}	
\bib{GP}{article}{
   author={G\"{u}rses, Metin},
   author={Pekcan, Asl\i},
   title={Nonlocal nonlinear Schr\"{o}dinger equations and their soliton
   solutions},
   journal={J. Math. Phys.},
   volume={59},
   date={2018},
   number={5},
   pages={051501, 17},
}		
\bib{HL16}{article}{
   author={Huang, Xin},
   author={Ling, Liming},
   title={Soliton solutions for the nonlocal nonlinear Schr\"{o}dinger 
   equation},
   journal={Eur. Phys. J. Plus},
   volume={131},
   date={2016},
   number={5},
   pages={148},
}		
\bib{IU86}{article}{
   author={Its, A. R.},
   author={Ustinov, A. F.},
   title={Time asymptotics of the solution of the Cauchy problem for the
   nonlinear Schr\"{o}dinger equation with boundary conditions of finite
   density type},
   language={Russian},
   journal={Dokl. Akad. Nauk SSSR},
   volume={291},
   date={1986},
   number={1},
   pages={91--95},
}
\bib{IU88}{article}{
   author={Its, A. R.},
   author={Ustinov, A. F.},
   title={Formulation of the scattering theory for the NLS equation with
   boundary conditions of finite density type in a soliton-free sector},
   language={Russian, with English summary},
   journal={Zap. Nauchn. Sem. Leningrad. Otdel. Mat. Inst. Steklov.
   (LOMI)},
   volume={169},
   date={1988},
   number={Voprosy Kvant. Teor. Polya i Statist. Fiz. 8},
   pages={60--67, 186--187},
   translation={
      journal={J. Soviet Math.},
      volume={54},
      date={1991},
      number={3},
      pages={900--905},
   },
}
\bib{J15}{article}{
   author={Jenkins, Robert},
   title={Regularization of a sharp shock by the defocusing nonlinear
   Schr\"{o}dinger equation},
   journal={Nonlinearity},
   volume={28},
   date={2015},
   number={7},
   pages={2131--2180},
}
\bib{LX15}{article}{
   author={Li, M.},
   author={Xu, T.},
   title={Dark and antidark soliton interactions in the nonlocal
   nonlinear Schr\"odinger equation with the self-induced 
   parity-time-symmetric potential},
   journal={Phys. Rev. E},
   volume={91},
   date={2015},
   pages={033202},
}		
\bib{Lou18}{article}{
   author={Lou, S. Y.},
   title={Alice-Bob systems, $\hat P$-$\hat T$-$\hat C$ symmetry invariant
   and symmetry breaking soliton solutions},
   journal={J. Math. Phys.},
   volume={59},
   date={2018},
   number={8},
   pages={083507, 20},
}	
\bib{LH17}{article}{
   author={Lou, S. Y.},
   author={Huang, F.},
   title={Alice-Bob physics, coherent solutions of nonlocal KdV systems},
   journal={Sci. Rep.},
   volume={7},
   date={2017},
   pages={869},
}	
\bib{MS20}{article}{
   author={Matveev, V. B.},
   author={Smirnov, A. O.},
   title={Multiphase solutions of nonlocal symmetric reductions of equations
   of the AKNS hierarchy: General analysis and simplest examples},
   language={Russian},
   journal={Teoret. Mat. Fiz.},
   volume={204},
   date={2020},
   number={3},
   pages={383--395},
}
\bib{MS18}{article}{
   author={Michor, J.},
   author={Sakhnovich, A. L.},
   title={GBDT and algebro-geometric approaches to explicit solutions and
   wave functions for nonlocal NLS},
   journal={J. Phys. A},
   volume={52},
   date={2019},
   number={2},
   pages={025201, 24},
}	
\bib{R21}{article}{
   author={Russo, Matthew},
   title={Local and nonlocal solitons in a coupled real system of
   Landau-Lifshitz equations},
   journal={Phys. D},
   volume={422},
   date={2021},
   pages={Paper No. 132893, 13},
}	
\bib{RS21PD}{article}{
   author={Rybalko, Yan},
   author={Shepelsky, Dmitry},
   title={Asymptotic stage of modulation instability for the nonlocal
   nonlinear Schr\"{o}dinger equation},
   journal={Phys. D},
   volume={428},
   date={2021},
   pages={Paper No. 133060, 30},
}
\bib{RSs}{article}{
   author={Rybalko, Yan},
   author={Shepelsky, Dmitry},
   title={Long-time asymptotics for the nonlocal nonlinear Schr\"{o}dinger
   equation with step-like initial data},
   journal={J. Differential Equations},
   volume={270},
   date={2021},
   pages={694--724},
}
\bib{RS21CIMP}{article}{
   author={Rybalko, Yan},
   author={Shepelsky, Dmitry},
   title={Long-time asymptotics for the integrable nonlocal focusing
   nonlinear Schr\"{o}dinger equation for a family of step-like initial data},
   journal={Commun. Math. Phys.},
   volume={382},
   date={2021},
   number={1},
   pages={87--121},
}
\bib{San18}{article}{
   author={Santini, P. M.},
   title={The periodic Cauchy problem for PT-symmetric NLS, I: the first
   appearance of rogue waves, regular behavior or blow up at finite times},
   journal={J. Phys. A},
   volume={51},
   date={2018},
   number={49},
   pages={495207, 21},
}	
\bib{SMMC}{article}{
   author={Sarma, A.},
   author={Miri, M.},
   author={Musslimani, Z.},
   author={Christodoulides, D.},
   title={Continuous and discrete Schr\"odinger systems with parity-time-symmetric 
   nonlinearities},
   journal={Phys. Rev. E},
   volume={89},
   date={2014},
   pages={052918},
}		
\bib{V02}{article}{
   author={Vartanian, A. H.},
   title={Long-time asymptotics of solutions to the Cauchy problem for the
   defocusing nonlinear Schr\"{o}dinger equation with finite-density initial
   data. II. Dark solitons on continua},
   journal={Math. Phys. Anal. Geom.},
   volume={5},
   date={2002},
   number={4},
   pages={319--413},
}
\bib{XCLM19}{article}{
   author={Xu, Tao},
   author={Chen, Yang},
   author={Li, Min},
   author={Meng, De-Xin},
   title={General stationary solutions of the nonlocal nonlinear Schr\"{o}dinger
   equation and their relevance to the $\C{PT}$-symmetric system},
   journal={Chaos},
   volume={29},
   date={2019},
   number={12},
   pages={123124, 12},
}	
\bib{YY20}{article}{
   author={Yang, Bo},
   author={Yang, Jianke},
   title={On general rogue waves in the parity-time-symmetric nonlinear
   Schr\"{o}dinger equation},
   journal={J. Math. Anal. Appl.},
   volume={487},
   date={2020},
   number={2},
   pages={124023, 23},
}
\bib{ZS73}{article}{
   author={Zakharov, V. E.},
   author={Shabat, A. B.},
   title={Interaction between solitons in a stable medium},
   journal={Soviet Physics JETP},
   volume={37},
   date={1973},
   pages={823--828},
}
\end{biblist}
\end{bibdiv}
\end{document}